\documentclass[11pt,reqno]{amsart}

\usepackage{color}
\usepackage[colorlinks=true, allcolors=blue,hypertexnames=false]{hyperref}
\usepackage{amsmath, amssymb, amsthm}
\usepackage{mathrsfs}
\usepackage{mathtools}
\usepackage{aliascnt}
\usepackage[noabbrev,capitalize,nameinlink]{cleveref}
\crefname{equation}{}{}
\usepackage{fullpage}
\usepackage[noadjust]{cite}
\usepackage{graphicx}
\usepackage{tikz}
\usetikzlibrary{arrows.meta,positioning,calc}
\usepackage{array}
\usepackage{float}
\usepackage{bbm}
\usepackage[T1]{fontenc}
\usepackage{url}
\usepackage{todonotes}
\allowdisplaybreaks[1]

\AtBeginDocument{\crefname{equation}{}{}}

\numberwithin{equation}{section}
\newtheorem{theorem}{Theorem}[section]
\crefname{theorem}{Theorem}{Theorems}

\newaliascnt{proposition}{theorem}
\newtheorem{proposition}[proposition]{Proposition}
\aliascntresetthe{proposition}
\crefname{proposition}{Proposition}{Propositions}

\newaliascnt{lemma}{theorem}
\newtheorem{lemma}[lemma]{Lemma}
\aliascntresetthe{lemma}
\crefname{lemma}{Lemma}{Lemmas}

\newaliascnt{claim}{theorem}

\aliascntresetthe{claim}
\crefname{claim}{Claim}{Claims}

\newaliascnt{corollary}{theorem}

\aliascntresetthe{corollary}
\crefname{corollary}{Corollary}{Corollaries}

\newaliascnt{conjecture}{theorem}

\aliascntresetthe{conjecture}
\crefname{conjecture}{Conjecture}{Conjectures}

\newtheorem*{question*}{Question}

\newaliascnt{fact}{theorem}

\aliascntresetthe{fact}
\crefname{fact}{Fact}{Facts}

\theoremstyle{definition}
\newaliascnt{definition}{theorem}

\aliascntresetthe{definition}
\crefname{definition}{Definition}{Definitions}

\newaliascnt{problem}{theorem}

\aliascntresetthe{problem}
\crefname{problem}{Problem}{Problems}

\newaliascnt{question}{theorem}

\aliascntresetthe{question}
\crefname{question}{Question}{Questions}

\newtheorem*{definition*}{Definition}

\newaliascnt{example}{theorem}

\aliascntresetthe{example}
\crefname{example}{Example}{Examples}

\newaliascnt{setup}{theorem}

\aliascntresetthe{setup}
\crefname{setup}{Setup}{Setups}

\theoremstyle{remark}
\newtheorem*{remark}{Remark}

\newcommand{\abs}[1]{\lvert#1\rvert}
\newcommand{\norm}[1]{\lVert#1\rVert}

\newcommand{\floor}[1]{\lfloor #1 \rfloor}
\newcommand{\ceil}[1]{\lceil #1 \rceil}
\newcommand{\paren}[1]{\bigg( #1 \bigg)}

\newcommand{\one}{\mathbbm{1}}

\newcommand{\mb}{\mathbb}

\newcommand{\mc}{\mathcal}

\newcommand{\mr}{\mathrm}

\newcommand{\on}{\operatorname}

\newcommand{\ip}[2]{\left\langle #1,#2\right\rangle}

\allowdisplaybreaks
\title{Total variation cutoff for Kac's walk on the sphere}
\author{Vishesh Jain}
\address{Department of Mathematics, Statistics, and Computer Science, University of Illinois Chicago, Chicago, IL 60607, USA}
\email{visheshj@uic.edu}

\author{Clayton Mizgerd}
\address{Department of Mathematics, Statistics, and Computer Science, University of Illinois Chicago, Chicago, IL 60607, USA}
\email{cmizge2@uic.edu}
\date{}

\begin{document}

\begin{abstract}
We prove cutoff in total variation distance for the discrete-time Kac walk on
\(S^{n-1}\) started from a coordinate vector.  The cutoff occurs at
\(
        C_{\mr{BRW}}n\log n,
\)
where \(C_{\mr{BRW}} \approx 3.8916\) is an explicit constant determined by the speed of the leftmost particle in a branching random walk.  In particular, the cutoff location is not at the conjectured time \( 2n\log n\).
\end{abstract}

\maketitle

\tableofcontents

\section{Introduction}\label{sec:introduction}
Let
\[
        S^{n-1}=\{x\in\mb R^n:\norm{x}_2=1\}
\]
denote the Euclidean unit sphere in $\mb{R}^n$ and let \(\sigma_{n-1}\) be the Haar probability measure on \(S^{n-1}\).
The discrete-time Kac walk \((X_t)_{t\geq 0}\) has transition kernel \(P_{n-1}\)
defined as follows.  From \(x\in S^{n-1}\), choose an unordered pair
\(\{I_t,J_t\}\subset[n]\) uniformly from the \(\binom n2\) coordinate
pairs, choose \(\Theta_t\sim\on{Unif}[0,2\pi)\) independently, and rotate
\(x\) by angle \(\Theta_t\) in the \((I_t,J_t)\)-coordinate plane.  Recall that the total variation distance between two measures $\mu$ and $\nu$ on the same space $(\Omega, \mc F)$ is defined by
\[
        \norm{\mu-\nu}_{\on{TV}}
        :=
        \sup_{A \in \mc F}\abs{\mu(A)-\nu(A)}.
\]
Let
\[
        d_n^{(1)}(t)
        :=
        \norm{P_{n-1}^t(e_1,\cdot)-\sigma_{n-1}}_{\on{TV}}
\]
denote the total variation distance between the law of the $t$-step Kac walk started from $e_1 = (1,0,\dots,0) \in S^{n-1}$ and $\sigma_{n-1}$ and, for \(\varepsilon\in(0,1)\), let
\[
        t_{n,1}(\varepsilon)
        :=
        \min\{t\ge0:d_n^{(1)}(t)\le\varepsilon\}
\]
denote the corresponding $\varepsilon$-mixing time from a coordinate start. 

\medskip

Our main result shows that the Kac walk started from a coordinate vector exhibits cutoff in total variation distance. The constant appearing in the statement of the theorem is discussed in \cref{sub:cutoff-constant}.

\begin{theorem}\label{thm:main-cutoff}
There is an explicit constant $C_{\mr{BRW}} \approx 3.8916$ 
such that, for every fixed \(c>0\),
\begin{align}
        &d_n^{(1)}(\ceil{c n\log n})\longrightarrow1,
        &&0<c<C_{\mr{BRW}},\label{eq:main-lower}\\
        &d_n^{(1)}(\ceil{c n\log n})\longrightarrow0,
        &&c>C_{\mr{BRW}}.\label{eq:main-upper}
\end{align}
\end{theorem}

\begin{remark}
Our proof shows that $d_{n}^{(1)}(\lceil C_{\mr{BRW}}n\log n + s n\rceil) \to 0$ as $s \to \infty$. We believe that, with additional effort, the lower bound in \cref{eq:main-lower} can be refined to show that $d_{n}^{(1)}(\lceil C_{\mr{BRW}}n\log n - s n \log \log n \rceil) \to 1$ as $s \to \infty$. For simplicity of presentation, we do not pursue this refinement here. We also conjecture that the upper bound in \cref{eq:main-upper} holds from an arbitrary deterministic start.    
\end{remark}

The Kac walk has a long history, which we briefly summarize in \cref{sub:related-work}. Most relevant to this work, Pillai and Smith showed that
the total variation mixing time from a worst-case start lies between
\((1/2)n\log n\) and \(200n\log n\) \cite{PillaiSmith2017}.  They also
conjectured cutoff at time \(2n\log n\).  Our main result \cref{thm:main-cutoff} establishes cutoff from a coordinate start, refutes the conjectured constant, and gives
the correct constant.

\subsection{The cutoff constant}\label{sub:cutoff-constant}

The constant is determined
by the speed of the minimum position in a binary branching random walk. The reason that this same constant appears in both the upper and lower bounds is explained in \cref{sub:overview}.

Let \(U\sim\on{Beta}(1/2,1/2)\). Consider the continuous-time
branching random walk started from one particle at \(0\), in which a
particle at \(x\) branches at exponential rate \(1\) into two children at
\[
        x-\log U,
        \qquad
        x-\log(1-U).
\]
Let \(B_T\) denote the minimum position among particles alive at time
\(T\). We compute the almost sure speed of $B_T$. 

We use Biggins's theorem for general branching random walks
\cite[Corollary~2]{Biggins1995}. We very briefly describe Biggins's setup, referring the reader to \cite{Biggins1995} for more details. In Biggins's notation, the reproduction
law is described by a point process on space-time with mean measure
\(\mu\), and one defines
\[
        M(\theta,\phi)
        :=
        \int e^{-\theta x-\phi t}\,\mu(dx,dt),
        \qquad
        \alpha(\theta)
        :=
        \inf\{\phi:M(\theta,\phi)\le1\}.
\]
Biggins shows that the almost sure speed of the rightmost
particle is
\[
        \inf\{a:\alpha^*(a)<0\},
\]
where
\[
        \alpha^*(a)
        :=
        \inf_{\theta<0}\{a\theta+\alpha(\theta)\}.
\]

We apply this theorem to the reflected process \(X=-V\), where \(V\) is
the logarithmic-loss position above.  In the reflected process, a particle at
\(x\) gives birth, after an independent exponential time \(E\) of rate
\(1\), to two children at
\[
        x+\log U,
        \qquad
        x+\log(1-U).
\]
Thus the reproduction point process is
\[
        \delta_{(\log U,E)}
        +
        \delta_{(\log(1-U),E)}.
\]
For \(\lambda>0\), direct computation gives
\[
\begin{aligned}
        M(-\lambda,\phi)
        &=
        \mb E\left[
        e^{\lambda\log U-\phi E}
        +
        e^{\lambda\log(1-U)-\phi E}
        \right]  \\
        &=
        \frac{m(\lambda)}{1+\phi},
\end{aligned}
\]
where
\begin{equation}\label{eq:m-lambda-main}
        m(\lambda)
        :=
        \mb E[U^\lambda+(1-U)^\lambda]
        =
        \frac{2\Gamma(\lambda+1/2)}
             {\sqrt\pi\,\Gamma(\lambda+1)}.
\end{equation}
Hence \(\alpha(-\lambda)=m(\lambda)-1\), and
\[
        \alpha^*(a)
        =
        \inf_{\lambda>0}\{-a\lambda+m(\lambda)-1\}.
\]
It follows that
\[
        \inf\{a:\alpha^*(a)<0\}
        =
        -\sup_{\lambda>0}\frac{1-m(\lambda)}{\lambda}.
\]
Since the rightmost particle in the reflected process is \(-B_T\),
Biggins's theorem gives
\[
        \frac{B_T}{T}
        \longrightarrow
        \gamma_{\mr{BRW}}
        :=
        \sup_{\lambda>0}\frac{1-m(\lambda)}{\lambda}
        =
        \sup_{\lambda>1}\frac{1-m(\lambda)}{\lambda}
\]
almost surely.  The last equality uses \(m(1)=1\), \(m(\lambda)\ge1\) on
\((0,1]\), and \(m(\lambda)<1\) on \((1,\infty)\).

In our case, the convergence also holds in \(L^1\). Since \(B_T/T\to\gamma_{\mr{BRW}}\) almost surely, it remains only to
check uniform integrability of \(\{B_T/T:T\ge1\}\). To see this, at each
branch keep the child with larger mass.  Along this selected path, each
branch increases the logarithmic loss by at most \(\log2\), and the
number of branches by time \(T\) is a rate-one Poisson random variable
\(N_T\).  Therefore
\[
        0\le B_T\le(\log2)N_T,
\]
and
\[
        \sup_{T\ge1}\mb E[(B_T/T)^2]
        \le
        (\log2)^2\sup_{T\ge1}\mb E[(N_T/T)^2]
        \le
        2(\log2)^2.
\]
This proves uniform integrability.

\medskip

The cutoff constant is
\[
        C_{\mr{BRW}}
        :=
        \frac1{2\gamma_{\mr{BRW}}} =     \inf_{\lambda>1}
        \frac{\lambda}{2(1-m(\lambda))}.
\]
Numerically (see \cref{app:numerical}),
\[
        \gamma_{\mr{BRW}}\approx0.12848,
        \qquad
        C_{\mr{BRW}}\approx3.8916.
\]
The factor \(2\) comes from the fact that each update of the Kac walk
selects two coordinates, so that  \(c n\log n\) discrete
updates correspond to \((2+o(1))c\log n\) units of branching time.

\subsection{Background and related work}\label{sub:related-work}

Kac introduced his walk as a finite-particle model for the approach to
equilibrium in kinetic theory \cite{Kac1956}.  In this model, the
coordinates represent one-dimensional particle velocities and the
constraint \(X\in S^{n-1}\) expresses conservation of kinetic energy.
Understanding the rate of convergence to equilibrium for this walk is
therefore a quantitative form of Kac's program.

There is a large literature on Kac-type collision processes.
Much of the early progress concerned spectral gaps, entropy production,
and propagation of chaos.  For instance, the spectral gap for Kac's
model was studied by Janvresse, Carlen--Carvalho--Loss, Maslen, and
Caputo, among others
\cite{Janvresse2001,CarlenCarvalhoLoss2003,Maslen2003,Caputo2008}.
Entropy and propagation-of-chaos estimates were developed further in
important works of Carlen--Carvalho--Le Roux--Loss--Villani, and of
Mischler--Mouhot
\cite{CarlenCarvalhoLeRouxLossVillatoro2010,MischlerMouhot2013}.  These
results are central to the kinetic-theory side of the subject, but they
do not by themselves say anything non-trivial about total variation mixing from a point mass, since the initial law is singular with respect to the stationary distribution, and its
\(L^2\)-distance from stationarity is infinite.

Diaconis and Saloff-Coste were the first to obtain a quantitative bound on the total variation mixing time of the Kac walk \cite{DiaconisSaloffCoste2000}.  Jiang gave
the first polynomial upper bound on the total variation mixing time from a coordinate
start, proving mixing in \(O(n^5(\log n)^3)\) steps
\cite{Jiang2012}.  Pillai and Smith determined the asymptotic order of the total variation mixing time from a worst-case start, and in fact, showed pre-cutoff with window $[(1/2) n \log n, 200 n \log n]$. They also conjectured cutoff at time
\(2n\log n\).  Our main result \cref{thm:main-cutoff} proves cutoff 
from a coordinate start and, since \(C_{\mr{BRW}} > 2\), disproves the
conjectured cutoff location.

The lower bound \cref{eq:main-lower} is related to, but distinct from, the obstruction in the
(complete-graph) repeated-averages process.  In that model, a selected pair
of coordinates is replaced by the arithmetic average of the pair.  Chatterjee,
Diaconis, Sly and Zhang proved cutoff at
\(
        \frac{1}{2\log 2}\,n\log n
\)
starting from $e_1$ \cite{ChatterjeeDiaconisSlyZhang2022}; see also the work of
Movassagh, Szegedy and Wang on repeated averaging over general graphs
\cite{MovassaghSzegedyWang2024}. Caputo, Quattropani and Sau recently developed a unified Wasserstein
cutoff theory for mean-field exchange models, including stochastic
redistribution models \cite{CaputoQuattropaniSau2025}.  The
squared-coordinate (or energy) Kac chain studied here (see \cref{sub:squared-chain}) is also a stochastic
redistribution model, with redistribution variable
\(U\sim\on{Beta}(1/2,1/2)\). The Wasserstein theory in
\cite{CaputoQuattropaniSau2025} is governed by the evolution of a
size-biased coordinate.  In contrast, total variation from a coordinate
start is governed by whether any descendant of the initially nonzero
coordinate remains atypically large.  This extremal obstruction is what
leads here to the branching-random-walk constant \(C_{\mr{BRW}}\).

\subsection{Proof overview}\label{sub:overview}

We give an overview of the proof, with particular
attention to the appearance of \(C_{\mr{BRW}}\) in the two bounds.

The first step is an exact reduction to squared coordinates.  If a Kac
update acts in the \((i,j)\)-plane, then the energy in that plane is
split in \(\on{Beta}(1/2,1/2)\) proportions.  Consequently,
\[
        Y_t=(X_{t,1}^2,\ldots,X_{t,n}^2)
\]
is a Markov chain \(Q_n\) on the simplex
\[
        \Delta_{n-1}
        :=
        \left\{y\in[0,1]^n:\sum_{i=1}^n y_i=1\right\}
\]
with stationary law
\[
        \pi_n=\on{Dir}_n(1/2,\ldots,1/2).
\]
For the coordinate start, passing to squared coordinates preserves total
variation distance:
\[
        \norm{P_{n-1}^t(e_1,\cdot)-\sigma_{n-1}}_{\on{TV}}
        =
        \norm{Q_n^t(e^{(1)},\cdot)-\pi_n}_{\on{TV}}.
\]
This identity is proved in \cref{prop:sphere-energy}.  It therefore
suffices to study \(Q_n\) started from
\(e^{(1)}=(1,0,\ldots,0)\).

We first discuss the lower bound.  Let
\[
        L_t:=\max_{1\le i\le n}Y_{t,i}.
\]
For every fixed \(0<\beta<1\),
\[
        \pi_n\left(\max_i y_i\ge n^{-\beta}\right)
        \longrightarrow0;
\]
see \cref{lem:stationary-max}.  Thus it is enough to show that, before
the claimed cutoff time, a coordinate of \(Y_t\) is still larger than
\(n^{-\beta}\) with high probability.

Starting from \(e^{(1)}\), the initial unit of energy is repeatedly split
when a coordinate carrying part of that energy is selected.  A given
coordinate is selected with probability \(2/n\) at each update, so \(t\)
updates correspond to approximately \(2t/n\) units of time in the
branching random walk from \cref{sub:cutoff-constant}.  At
\(t=c n\log n\), the minimum logarithmic loss in that branching random
walk is therefore
\[
        \bigl(2c\gamma_{\mr{BRW}}+o(1)\bigr)\log n.
\]
Equivalently, the largest part produced by the repeated splittings has
the scale
\[
        \exp\left\{
        -\bigl(2c\gamma_{\mr{BRW}}+o(1)\bigr)\log n
        \right\}
        =
        n^{-2c\gamma_{\mr{BRW}}+o(1)}.
\]
This is larger than \(n^{-\beta}\) for some \(\beta<1\) precisely when
\[
        2c\gamma_{\mr{BRW}}<1,
        \qquad\text{or equivalently}\qquad
        c<C_{\mr{BRW}}.
\]

\medskip

For the upper bound, we use a different representation of the law of the energy chain (see
\cref{sub:scale-vector-representation}). For \(a=(a_1,\ldots,a_n)\in[0,\infty)^n\setminus\{0\}\), let \(F_{a}\) be
the law on \(\Delta_{n-1}\) of
\begin{equation*}
        \left(
        \frac{a_1G_1}{\sum_k a_kG_k},\ldots,
        \frac{a_nG_n}{\sum_k a_kG_k}
        \right),
        \qquad
        G_i\stackrel{\mr{iid}}{\sim}\on{Gamma}(1/2,1).
\end{equation*}
Then \(F_{\one}=\pi_n\), \(F_{e^{(1)}}=\delta_{e^{(1)}}\), and
\cref{prop:hidden-rep} gives
\[
        Q_n^t(e^{(1)},\cdot)=\mb E[F_{A_t}],
\]
where \(A_0=e^{(1)}\) and one step of \(A_t\) replaces a uniformly
selected pair by
\[
        A_i'=A_j'=BA_i+(1-B)A_j,
        \qquad
        B\sim\on{Beta}(1/2,1/2),
\]
leaving the other coordinates unchanged.

The identity \(F_{\lambda a}=F_a\) shows that only the projective class of \(A_t\) is relevant.  We therefore write
\[
        R_t:=\sum_iA_{t,i},
        \qquad
        \bar A_t:=\frac{R_t}{n},
        \qquad
        \xi_t:=A_t-\bar A_t\one,
        \qquad
        u_t:=\frac{A_t}{\bar A_t}-\one
             =\frac{n\xi_t}{R_t}.
\]
Thus \(F_{A_t}=F_{\one+u_t}\), and \(u_{t,i}\) is the relative
deviation of \(A_{t,i}\) from the average coordinate \(R_t/n\).
To compare \(F_{A_t}\) with \(F_{\one}=\pi_n\), we first show that
these relative deviations are uniformly small whenever \(R_t\) is
bounded away from zero.

More precisely, if
\(t_n=\ceil{c n\log n}\) with \(c>C_{\mr{BRW}}\), then
\cref{prop:Linfty-full} gives, for every fixed \(r,\varepsilon>0\),
\[
        \mb P\left(
        R_{t_n}\ge r,\ 
        \norm{u_{t_n}}_\infty>\varepsilon
        \right)
        \longrightarrow0.
\]
The restriction to \(R_{t_n}\ge r\) is needed because
\(u_t=n\xi_t/R_t\).  It can be removed at the end using
\cref{lem:R-tail-full}, which shows that
\[
        \limsup_{n\to\infty}\mb P(R_{t_n}\le r)
        \le C\sqrt r,
        \qquad 0<r\le1,
\]
and we eventually let \(r\downarrow0\).

We now explain why the threshold in this estimate is
\(C_{\mr{BRW}}\). For \(2<p<4\), set
\[
        \Phi_p(x):=\sum_i|x_i|^p,
        \qquad
        W_t:=\sum_i\xi_{t,i}^2.
\]
In \cref{prop:fractional-drift-full}, we prove the one-step drift estimate
\[
\begin{aligned}
        \mb E[\Phi_p(\xi_{t+1})\mid A_t]
        &\le
        \left(
        1-\frac{2(1-m(p))}{n}
        \right)\Phi_p(\xi_t)
        +
        C_p n^{-p/2}W_t^{p/2}.
\end{aligned}
\]
The coefficient \(2(1-m(p))/n\) can be read directly from the update rule.
The quantity
\[
        m(p)=\mb E[B^p+(1-B)^p]
\]
is the fraction of \(p\)-th moment retained by the beta weights.  Indeed,
when the two selected coordinates are treated separately, their
contribution to \(\Phi_p\) is multiplied, in expectation, by \(m(p)\).
Since each coordinate is selected with probability \(2/n\), the corresponding contribution to the one-step drift is
\[
        \left(1-\frac2n\right)\Phi_p(\xi_t)
        +
        \frac{2m(p)}n\Phi_p(\xi_t)
        =
        \left(1-\frac{2(1-m(p))}{n}\right)\Phi_p(\xi_t).
\]
The interaction between the two selected coordinates, together with the
change in their common average, accounts for the error term
\(
        C_p n^{-p/2}W_t^{p/2}.
\)
The same function \(m(p)\)
appears in the branching-random-walk calculation because a beta split
replaces a mass \(x\) by \(Bx\) and \((1-B)x\), whose total \(p\)-th
power has expectation
\[
        \mb E\left[(Bx)^p+((1-B)x)^p\right]
        =
        m(p)x^p.
\]
Ignoring for the moment the error term in the drift estimate, after
\(t=c n\log n\) steps we have
\[
        \mb E\Phi_p(\xi_t)
        \approx
        \exp\left\{-2(1-m(p))\frac{t}{n}\right\}
        =
        n^{-2c(1-m(p))}.
\]
On the event \(\{R_t\ge r\}\), the identity
\(u_t=n\xi_t/R_t\) and Markov's inequality give
\[
        \mb P\left(
        R_t\ge r,\ 
        \norm{u_t}_\infty>\varepsilon
        \right)
        \le
        \left(\frac{n}{r\varepsilon}\right)^p
        \mb E\Phi_p(\xi_t).
\]
Thus the contribution of the initial condition tends to zero provided
\[
        p-2c(1-m(p))<0,
        \qquad\text{or equivalently}\qquad
        c>\frac{p}{2(1-m(p))}.
\]
Optimizing this
over \(p\) leads to the appearance of
\[
        \inf_{p>1}\frac{p}{2(1-m(p))}
        =
        C_{\mr{BRW}}.
\]

\medskip

It remains to upgrade to total-variation convergence. This does not require running the chain for any additional time. We work at the same time \(t_n=\ceil{c n\log n}\) and compare \[ Q_n^{t_n}(e^{(1)},\cdot) = \mb E[F_{\one+u_{t_n}}] \] directly with \(\pi_n=F_{\one}\). We bound total variation through \(\chi^2\)-divergence, using the gamma representation of the measures \(F_{\one+u}\). The \(\ell^\infty\)-estimate for \(u_{t_n}\) is needed, but is not by itself sufficient in growing dimension. To see the issue, consider first a fixed vector \(u\). The gamma calculation in \cref{lem:product-gamma-full} gives \[ 1+\chi^2(F_{\one+u}\Vert\pi_n) \le \prod_{i=1}^n(1-u_i^2)^{-1/2}. \] When \(\norm{u}_\infty\) is small, the logarithm of the right-hand side has leading term \[ \frac12\sum_i u_i^2. \] Thus, even if every coordinate of \(u\) is small, their cumulative effect may not be. In particular, it is not enough to compare each individual law \(F_{\one+u}\) with \(\pi_n\).

The averaging over the random scale vector changes this calculation. After restricting to an event of probability tending to one on which \(R_{t_n}\) is bounded below, \(\norm{u_{t_n}}_\infty\le\delta\), and \(\norm{u_{t_n}}_2\) is controlled, let \(\mu_n\) denote the resulting conditioned mixture. If \(u\) and \(v\) are two independent copies of \(u_{t_n}\) under this conditioning, then \cref{lem:product-gamma-full} gives \[ 1+\chi^2(\mu_n\Vert\pi_n) \le \mb E \prod_{i=1}^n(1-u_iv_i)^{-1/2}. \] The self-overlap \(\sum_i u_i^2\) from the fixed-\(u\) calculation is therefore replaced by the cross-overlap \(\sum_i u_iv_i\) of two independent copies. More precisely, the \(\ell^\infty\)-bound gives \[ \log\prod_{i=1}^n(1-u_iv_i)^{-1/2} \le \frac12\sum_i u_iv_i + \frac{1}{2}\sum_i u_i^2v_i^2. \]
Thus, to first order, the product $\prod_i(1-u_i v_i)^{-1/2}$ measures the coordinatewise alignment between the two perturbations through the cross-overlap \(\sum_i u_iv_i\). It remains to show that two independent copies are unlikely to align substantially.

This is where the coordinate start enters the upper bound. Since \(A_0=e^{(1)}\), the law of \(u_{t_n}\) is invariant under permutations of coordinates \(2,\ldots,n\). Conditioned on \(u_1,v_1\) and on the two multisets \[ \{u_2,\ldots,u_n\}, \qquad \{v_2,\ldots,v_n\}, \] the relative matching of the remaining coordinates is therefore a uniform random permutation. If \(\Pi\) denotes this permutation, then \[ \mb E_\Pi\sum_{i=2}^n u_i v_{\Pi(i)} = \frac{u_1v_1}{n-1}, \] since \(\sum_i u_i=\sum_i v_i=0\). The estimate in \cref{lem:permutation-full} gives \[ \begin{aligned} &\mb E_\Pi\exp\left\{ \frac12\sum_{i=2}^n u_i v_{\Pi(i)} + \frac12\sum_{i=2}^n u_i^2v_{\Pi(i)}^2 \right\} \\ &\qquad\le \exp\left\{ \frac{u_1v_1}{2(n-1)} + C\frac{\norm{u}_2^2\norm{v}_2^2}{n} \right\}. \end{aligned} \] Including the first coordinate and using \(\abs{u_1},\abs{v_1}\le\delta\), we obtain \[ \mb E \prod_{i=1}^n(1-u_iv_i)^{-1/2} \le \exp\left\{ C\delta^2 + C\frac{\norm{u}_2^2\norm{v}_2^2}{n} \right\}. \]

Finally, on the event under consideration,
\[
        \norm{u}_2^2,\norm{v}_2^2
        \le
        K\frac{n^2\rho_n^{t_n}}{r^2},
\]
and hence
\[
        \frac{\norm{u}_2^2\norm{v}_2^2}{n}
        \le
        \frac{K^2}{r^4}n^3\rho_n^{2t_n}
        =
        O_{r,K}\bigl(n^{3-c+o(1)}\bigr),
\]
since \(\rho_n=1-\frac1{2n}+O(n^{-2})\).  This tends to zero for
\(c>3\), so
\[
        \limsup_{n\to\infty}\chi^2(\mu_n\Vert\pi_n)
        \le
        e^{C\delta^2}-1.
\]
Letting \(\delta\downarrow0\), \(K\to\infty\), and \(r\downarrow0\)
gives total-variation convergence.

\subsection{Notation}\label{sub:notation}

We write \([n]=\{1,\ldots,n\}\), \(S_m\) for the symmetric group on
\([m]\), and all logarithms are natural.  The vectors
\(e_1,\ldots,e_n\) are the standard basis vectors of
\(\mb R^n\).  When the vector \((1,0,\ldots,0)\) is regarded as a point
of the simplex, we write it as \(e^{(1)}\).  The symbol \(\one\) denotes
the all-ones vector, with dimension determined by context, while
\(\one_E\) denotes the indicator of an event or set \(E\).  We write
\[
        \Delta_{n-1}
        :=
        \left\{y\in[0,1]^n:\sum_{i=1}^n y_i=1\right\}.
\]
For \(x\in\mb R^n\) and \(1\le p<\infty\),
\[
        \norm{x}_p
        :=
        \left(\sum_{i=1}^n|x_i|^p\right)^{1/p},
        \qquad
        \norm{x}_\infty
        :=
        \max_{1\le i\le n}|x_i|,
\]
and \(\ip{x}{y}\) denotes the Euclidean inner product.

For a point \(x\), \(\delta_x\) denotes the Dirac probability measure at
\(x\).  If \(f\) is measurable and \(\mu\) is a measure, then
\(f_\#\mu\) denotes the push-forward of \(\mu\) under \(f\).  If \(M\)
is a random probability measure, then \(\mb E[M]\) denotes its mixture:
\[
        \mb E[M](A):=\mb E[M(A)]
\]
for every measurable set \(A\).  Subscripts on \(\mb P\) and \(\mb E\)
indicate either the initial state or the random variables with respect
to which probability or expectation is taken.  We use
\(\stackrel d=\) for equality in distribution and \(\Rightarrow\) for
convergence in distribution.  The notation
\(\stackrel{\mr{iid}}{\sim}\) and \(\stackrel{\mr{ind}}{\sim}\) indicates
identically distributed independent variables and independent variables,
respectively.

We use \(\on{Gamma}(\alpha,\theta)\) for the gamma distribution with
shape parameter \(\alpha>0\) and scale parameter \(\theta>0\), whose
density is
\[
        x\longmapsto
        \frac{x^{\alpha-1}e^{-x/\theta}}
             {\Gamma(\alpha)\theta^\alpha}\one_{\{x>0\}}.
\]
Here \(\Gamma\) denotes the gamma function.
We use \(\on{Beta}(\alpha,\beta)\), \(\on{Bin}(m,p)\), and
\(\on{Unif}(I)\) for the beta, binomial, and uniform distributions,
respectively.  The notation
\(\on{Dir}_n(\alpha_1,\ldots,\alpha_n)\) denotes the Dirichlet
distribution on \(\Delta_{n-1}\).  Equivalently, if
\(G_1,\ldots,G_n\) are independent with
\(G_i\sim\on{Gamma}(\alpha_i,1)\), then
\[
        \left(
        \frac{G_1}{\sum_kG_k},\ldots,
        \frac{G_n}{\sum_kG_k}
        \right)
        \sim
        \on{Dir}_n(\alpha_1,\ldots,\alpha_n).
\]

Unless another limit is explicitly specified, asymptotic statements are
as \(n\to\infty\).  For real-valued sequences \(f=f(n)\) and \(g=g(n)\),
with \(\abs{g}>0\) eventually, the relation \(f=O(g)\) means that
\(\abs{f}\le C\abs{g}\) for all sufficiently large \(n\), while
\(f=o(g)\) means that \(f/g\to0\).  We write
\(f\asymp g\) when both \(f=O(g)\) and \(g=O(f)\).  A subscript records
permitted dependence of the implicit constant; for example,
\(O_T(g)\) may depend on \(T\) but not on \(n\).  Constants denoted by
\(C\), with or without subscripts, are finite and positive and may
change from line to line.  

\subsection{Acknowledgments} The authors used ChatGPT extensively, at the level of a co-author, for brainstorming, help with proofs, literature review, checking for mistakes, writing code, and for preparing the manuscript. The mathematical content, the final text, and any remaining errors are the responsibility of the authors. 

V.J.~thanks Natesh Pillai and Aaron Smith for introducing him to this problem and for useful discussions, and Mehtaab Sawhney for helpful comments on an early version of the manuscript. V.J.~is partially supported by NSF grant DMS-2237646. C.M.~is partially supported by a Simons Dissertation Fellowship.

\section{Auxiliary processes}\label{sec:auxiliary}

Our proofs use two reductions. The first replaces Kac's walk on the sphere by
the squared-coordinate (or energy) chain on the simplex.  The second writes the law
of this chain as a mixture of normalized gamma laws with random scale
parameters. Throughout the paper, \(\on{Gamma}(\alpha,\theta)\) denotes the gamma
law with shape parameter \(\alpha>0\) and scale parameter \(\theta>0\),~i.e.~with density
\[
        x\mapsto
        \frac{1}{\Gamma(\alpha)\theta^\alpha}
        x^{\alpha-1}e^{-x/\theta}\one_{\{x>0\}} .
\]

\subsection{The energy chain}
\label{sub:squared-chain}

Let \(\Theta\sim\on{Unif}[0,2\pi)\).  Then
\[
        \cos^2\Theta\sim\on{Beta}(1/2,1/2).
\]
Thus, if a Kac update acts on coordinates \(i,j\), and if
\[
        s=x_i^2+x_j^2
\]
is the energy in the selected two-plane, then the squared coordinates
after the update have law
\[
        (X_{t+1,i}^2,X_{t+1,j}^2)
        \stackrel d=
        (sU,s(1-U)),
        \qquad
        U\sim\on{Beta}(1/2,1/2).
\]
The squared-coordinate (or energy) process
\[
        Y_t=(X_{t,1}^2,\ldots,X_{t,n}^2)
\]
is therefore a Markov chain on
\[
        \Delta_{n-1}
        :=
        \left\{y\in[0,1]^n:\sum_{i=1}^n y_i=1\right\}.
\]
We write \(Q_n\) for its transition kernel. Explicitly, from \(y\in\Delta_{n-1}\),
choose an unordered pair \(\{i,j\}\subset[n]\) uniformly and choose
\(U\sim\on{Beta}(1/2,1/2)\) independently.  Then set the next step to $y'$ defined by
\[
        y_i'=U(y_i+y_j),
        \qquad
        y_j'=(1-U)(y_i+y_j),
        \qquad
        y_k'=y_k\quad(k\notin\{i,j\}).
\]

We claim that \(Q_n\) is reversible with respect to
\[
        \pi_n:=\on{Dir}_n(1/2,\ldots,1/2).
\]
It will be convenient to work with the following equivalent representation of the Dirichlet distribution: if \(G_1,\ldots,G_n\) are independent
\(\on{Gamma}(1/2,1)\) variables, then
\[
        \left(\frac{G_1}{\sum_kG_k},\ldots,
              \frac{G_n}{\sum_kG_k}\right)
        \sim \pi_n .
\]
This representation also shows that \(\pi_n\) is the law of the squared
coordinates of a Haar-distributed point on \(S^{n-1}\).

To see reversibility, we recall the standard beta-gamma algebra. If
\(G_i\sim\on{Gamma}(\alpha_i,1)\) and
\(G_j\sim\on{Gamma}(\alpha_j,1)\) are independent, then
\[
        \frac{G_i}{G_i+G_j}
        \sim
        \on{Beta}(\alpha_i,\alpha_j),
        \qquad
        G_i+G_j\sim\on{Gamma}(\alpha_i+\alpha_j,1),
\]
and these two random variables are independent.  Consequently, if
\[
        Y\sim\on{Dir}_n(\alpha_1,\ldots,\alpha_n),
\]
then, conditioned on
\[
        \mc G_{ij}
        :=
        \sigma\bigl((Y_k)_{k\notin\{i,j\}}\bigr),
\]
one has
\[
        (Y_i,Y_j)
        \stackrel d=
        (B(Y_i+Y_j)B,(1-B)(Y_i+Y_j)),
        \qquad
        B\sim\on{Beta}(\alpha_i,\alpha_j),
\]
with \(B\) independent of \(\mc G_{ij}\).  Taking
\(\alpha_1=\cdots=\alpha_n=1/2\), the \(\{i,j\}\)-update of \(Q_n\)
resamples \((Y_i,Y_j)\) from this conditional law, leaving
\[
        \bigl((Y_k)_{k\notin\{i,j\}},\,Y_i+Y_j\bigr)
\]
fixed.  Hence each pair update is reversible with respect to \(\pi_n\),
and so is their average \(Q_n\).

The energy chain converges to $\pi_n$ from a deterministic
start.  This follows from the corresponding convergence theorem for
Kac's walk on the sphere. Indeed, let
\[
        \rho:S^{n-1}\to\Delta_{n-1},
        \qquad
        \rho(x)=(x_1^2,\ldots,x_n^2).
\]
For any \(x\in S^{n-1}\),
\[
        \rho_\#P_{n-1}^t(x,\cdot)= Q_n^t(\rho(x),\cdot),
        \qquad
        \rho_\#\sigma_{n-1}=\pi_n .
\]
Therefore, if \(P_{n-1}^t(x,\cdot)\Rightarrow\sigma_{n-1}\) for every
deterministic \(x\in S^{n-1}\), then, for every deterministic
\(y\in\Delta_{n-1}\),
\[
        Q_n^t(y,\cdot)\Rightarrow\pi_n .
\]

From a coordinate start, passing to squared coordinates loses no
total-variation information.

\begin{proposition}\label{prop:sphere-energy}
For every \(t\ge0\),
\begin{equation*}
        \norm{P_{n-1}^t(e_1,\cdot)-\sigma_{n-1}}_{\on{TV}}
        =
        \norm{Q_n^t(e^{(1)},\cdot)-\pi_n}_{\on{TV}} .
\end{equation*}
\end{proposition}

\begin{proof}
By the definition of the energy chain,
\[
        \rho_\#P_{n-1}^t(e_1,\cdot)=Q_n^t(e^{(1)},\cdot),
        \qquad
        \rho_\#\sigma_{n-1}=\pi_n .
\]
By contraction of total variation under measurable maps,
\[
        \norm{Q_n^t(e^{(1)},\cdot)-\pi_n}_{\on{TV}}
        \le
        \norm{P_{n-1}^t(e_1,\cdot)-\sigma_{n-1}}_{\on{TV}} .
\]

It remains to prove the reverse inequality.  For
\(y\in\Delta_{n-1}\setminus\{e^{(1)}\}\), let \(S(y,\cdot)\) be the
probability measure on
\[
        \rho^{-1}(y)
        =
        \{x\in S^{n-1}:x_i^2=y_i\text{ for all }i\}
\]
obtained by assigning independent symmetric signs to the nonzero
coordinates \(\sqrt{y_i}\), and setting the zero coordinates equal to
zero.  At \(y=e^{(1)}\), set
\[
        S(e^{(1)},\cdot)=\delta_{e_1},
\]
rather than the symmetric law on $\{\pm e_1\}$, in order to be consistent with our starting configuration $e_1$ for the Kac walk. If \(\nu\) is a probability measure on \(\Delta_{n-1}\), write
\[
        \nu S(A):=\int_{\Delta_{n-1}}S(y,A)\,\nu(dy),
        \qquad A\subseteq S^{n-1}\text{ measurable}.
\]
Then
\begin{equation}\label{eq:haar-energy-factorization}
        \sigma_{n-1}=\pi_nS .
\end{equation}

We claim that, for every \(t\ge0\),
\begin{equation}\label{eq:kac-energy-factorization}
        P_{n-1}^t(e_1,\cdot)
        =
        (Q_n^t(e^{(1)},\cdot)) S .
\end{equation}
The identity holds at \(t=0\).  Assume it holds at time \(t\), and
condition on \(\rho(X_t)=y\).  By the induction hypothesis, \(X_t\) has
law \(S(y,\cdot)\) on the fiber over \(y\).

Suppose the next chosen pair is \(\{i,j\}\), and write
\(s=y_i+y_j\).  Conditioned on the coordinates outside \(\{i,j\}\), the
uniform Kac angle replaces the selected two coordinates by a uniform
point on the circle of radius \(\sqrt{s}\).  Hence
\[
        (X_{t+1,i}^2,X_{t+1,j}^2)
        \stackrel d=
        (Us,(1-U)s),
        \qquad
        U\sim\on{Beta}(1/2,1/2),
\]
and all other squared coordinates are unchanged.  This is the
\(\{i,j\}\)-update of the energy chain.

We now identify the conditional law on the new fiber.  If \(s>0\), a
uniform point on the circle of radius \(\sqrt{s}\) has independent
symmetric signs conditioned on its coordinate squares.  If \(s=0\), both
selected coordinates remain zero.  The coordinates outside \(\{i,j\}\)
are unchanged and keep the conditional sign law from time \(t\).  Thus,
apart from \(e^{(1)}\), the conditional law of \(X_{t+1}\) given
\(\rho(X_{t+1})\) is \(S(\rho(X_{t+1}),\cdot)\).

At \(e^{(1)}\), the convention above gives the correct conditional law.
The only atom of \( Q_n^{t+1}(e^{(1)},\cdot) \) at \(e^{(1)}\) comes from
paths on which coordinate \(1\) has not been selected; on this event the
Kac walk is still at \(e_1\).  Hence the conditional law on this fiber is
\(S(e^{(1)},\cdot)=\delta_{e_1}\).  This proves
\eqref{eq:kac-energy-factorization} at time \(t+1\).

Finally, using \eqref{eq:kac-energy-factorization},
\eqref{eq:haar-energy-factorization}, and contraction of total variation
under the Markov kernel \(S\), we get
\[
        \norm{P_{n-1}^t(e_1,\cdot)-\sigma_{n-1}}_{\on{TV}}
        =
        \norm{Q_n^t(e^{(1)},\cdot)S-\pi_nS}_{\on{TV}}
        \le
        \norm{Q_n^t(e^{(1)},\cdot)-\pi_n}_{\on{TV}} .
\]
This proves the reverse inequality.
\end{proof}

\subsection{A scale-vector representation}
\label{sub:scale-vector-representation}

For the upper bound, we use a different representation of the law of the
energy chain.  For \(a=(a_1,\ldots,a_n)\in[0,\infty)^n\setminus\{0\}\),
let \(F_a\) be the law on \(\Delta_{n-1}\) of
\[
        \left(
        \frac{a_1G_1}{\sum_k a_kG_k},\ldots,
        \frac{a_nG_n}{\sum_k a_kG_k}
        \right),
        \qquad
        G_i\stackrel{\mr{iid}}{\sim}\on{Gamma}(1/2,1).
\]
Since
\(a\ne0\), the denominator is positive almost surely.  In particular,
\[
        F_{\one}=\pi_n,
        \qquad
        F_{e^{(1)}}=\delta_{e^{(1)}}.
\]
The law \(F_a\) depends only on the projective class of \(a\):
\[
        F_{\lambda a}=F_a,
        \qquad \lambda>0.
\]
For a random scale vector \(A\), \(\mb E[F_A]\) denotes the probability
measure obtained by averaging \(F_A\) over \(A\).

Define \((A_t)_{t\ge0}\) on \([0,\infty)^n\setminus\{0\}\) as follows.
Start from
\[
        A_0=e^{(1)}.
\]
At each step, choose an unordered pair \(\{i,j\}\subset[n]\) uniformly,
choose \(B\sim\on{Beta}(1/2,1/2)\) independently, and set
\begin{equation}\label{eq:hidden-update}
        A_i'=A_j'=BA_i+(1-B)A_j,
        \qquad
        A_k'=A_k\quad(k\notin\{i,j\}).
\end{equation}

The beta variable \(B\) in \eqref{eq:hidden-update} comes from the ratio
\(G_i/(G_i+G_j)\) in the proof below.  The transition of \(Q_n\) itself
uses an independent \(U\sim\on{Beta}(1/2,1/2)\).

\begin{proposition}
\label{prop:hidden-rep}
For every \(t\ge0\),
\begin{equation*}
        Q_n^t(e^{(1)},\cdot) =\mb E[F_{A_t}] .
\end{equation*}
\end{proposition}

\begin{proof}
It suffices to prove the one-step identity.  Fix
\(a\in[0,\infty)^n\setminus\{0\}\), and let
\[
        Z_k=a_kG_k,
        \qquad
        G_k\stackrel{\mr{ind}}{\sim}\on{Gamma}(1/2,1).
\]
Then \(Z/\sum_k Z_k\) has law \(F_a\).  If \(a_k=0\), the coordinate
\(Z_k\) is interpreted as identically zero.

Fix the selected pair \(\{i,j\}\).  Let \(U\sim\on{Beta}(1/2,1/2)\) be
independent of the variables above.  The \(Q_n\)-update on this pair is
the normalization of the replacement
\[
        Z_i'=U(Z_i+Z_j),
        \qquad
        Z_j'=(1-U)(Z_i+Z_j),
        \qquad
        Z_k'=Z_k \quad (k\notin\{i,j\}).
\]
Since \(Z_i+Z_j\) is preserved, the normalizing denominator is unchanged.

Now write
\[
        R=G_i+G_j,
        \qquad
        B=\frac{G_i}{G_i+G_j}, \qquad a^* = B a_i + (1-B) a_j.
\]

Notice that
\[ (Z_i', Z_j') = a^*(UR, (1-U)R). \]

By the beta-gamma algebra, $R \sim \on{Gamma}(1,1)$, and so by independence of $R$ and $U$,
\[ (Z_i',Z_j') \sim a^* (\on{Gamma}(1/2,1), \on{Gamma}(1/2,1) ). \]

Set
\[ a_i' = a_j' = a^*, \qquad a_k' = a_k \quad (k \ne i,j). \]

Notice that $a^* = Ba_i + (1-B)a_j$ is exactly the law of the averaging process $A_i$.

By induction, $Z_k' \sim a_k \on{Gamma}(1/2,1)$ for $k \ne i,j$.  Thus $Z'/\sum_k Z_k' \sim F_{a'}$.  Averaging over $i,j$ and iterating from $F_{e^{(1)}} = Q_n^0(e^{(1)},\cdot) = \delta_{e^{(1)}}$ gives the result. \qedhere


\end{proof}

\section{The lower bound}\label{sec:lower}

By \cref{prop:sphere-energy}, it is enough to prove the lower bound for
the energy chain started from \(e^{(1)}\).  We separate this chain from
stationarity by the event that some coordinate has energy at least
\(n^{-\beta}\).  Under the stationary law this event has vanishing
probability for every fixed \(\beta<1\).  Before time
\(C_{\mr{BRW}}n\log n\), however, we will show that a portion of the
initial unit mass remains above this scale with high probability.

\begin{lemma}\label{lem:stationary-max}
For every fixed \(0<\beta<1\),
\[
        \pi_n\left(\max_{1\le i\le n}y_i\ge n^{-\beta}\right)
        \longrightarrow0 .
\]
\end{lemma}

\begin{proof}
Under \(\pi_n\), for any $i$, \(y_i\sim\on{Beta}(1/2,(n-1)/2)\), with density
\[
        f_n(u)=
        \frac{\Gamma(n/2)}{\sqrt\pi\,\Gamma((n-1)/2)}
        u^{-1/2}(1-u)^{(n-3)/2},
        \qquad 0<u<1 .
\]
By the union bound and
\[
        \frac{\Gamma(n/2)}{\Gamma((n-1)/2)}\le C\sqrt n,
\]
we get
\[
\begin{aligned}
        \pi_n\left(\max_i y_i\ge n^{-\beta}\right)
        &\le
        Cn^{3/2}
        \int_{n^{-\beta}}^1
        u^{-1/2}(1-u)^{(n-3)/2}\,du .
\end{aligned}
\]
For \(a\in(0,1)\),
\[
        \int_a^1 u^{-1/2}(1-u)^{(n-3)/2}\,du
        \le
        \int_a^1 u^{-1/2}e^{-(n-3)u/2}\,du
        \le
        C n^{-1}a^{-1/2}e^{-(n-3)a/2}.
\]
Taking \(a=n^{-\beta}\) gives
\[
        \pi_n\left(\max_i y_i\ge n^{-\beta}\right)
        \le
        Cn^{(1+\beta)/2}
        \exp\left\{-\frac{n-3}{2}n^{-\beta}\right\}
        \longrightarrow0 . \qedhere
\]
\end{proof}

Let
\[
        L_t:=\max_{1\le i\le n}Y_{t,i}
\]
denote the largest coordinate energy at time \(t\).  The lower bound is
reduced to the following persistence estimate.

\begin{proposition}\label{prop:persistence}
Fix \(0<c<C_{\mr{BRW}}\).  There exists \(\beta\in(0,1)\) such that,
with \(t_n=\ceil{c n\log n}\),
\[
        \mb P_{e^{(1)}}\left(L_{t_n}\ge n^{-\beta}\right)
        \longrightarrow1 .
\]
\end{proposition}

\begin{proof}[Proof of the lower bound in \cref{thm:main-cutoff}]
Let \(0<c<C_{\mr{BRW}}\), and set \(t_n=\ceil{c n\log n}\).  By
\cref{prop:persistence}, there is \(\beta\in(0,1)\) such that
\[
        \mb P_{e^{(1)}}\left(\max_iY_{t_n,i}\ge n^{-\beta}\right)
        \longrightarrow1 .
\]
By \cref{lem:stationary-max},
\[
        \pi_n\left(\max_i y_i\ge n^{-\beta}\right)\longrightarrow0 .
\]
Therefore
\[
        \norm{Q_n^{t_n}(e^{(1)},\cdot)-\pi_n}_{\on{TV}}
        \longrightarrow1 .
\]
The reduction in \cref{prop:sphere-energy} gives the same lower bound for
the Kac walk on the sphere.
\end{proof}

It remains to prove \cref{prop:persistence}.  We use the branching random
walk introduced in \cref{sub:cutoff-constant}.  For convenience, recall
the definition.  The process starts with one particle at \(0\).  Each
particle branches at rate \(1\), independently of the others; when a
particle at \(x\) branches, it is replaced by two particles at
\[
        x-\log U,
        \qquad
        x-\log(1-U),
\]
where \(U\sim\on{Beta}(1/2,1/2)\).  Let \(\mc A_T\) denote the set of
particles present at time \(T\).  For \(u\in\mc A_T\), write \(V_u(T)\)
for the position of \(u\) at time \(T\), and set
\[
        B_T:=\min_{u\in\mc A_T} V_u(T).
\]
As discussed in \cref{sub:cutoff-constant},
\begin{equation}\label{eq:brw-speed}
        \frac{B_T}{T}\xrightarrow[T\to\infty]{}
        \gamma_{\mr{BRW}}
        :=
        \sup_{\lambda>0}\frac{1-m(\lambda)}{\lambda}
\end{equation}
almost surely and in \(L^1\), where
\[
        m(\lambda)
        :=
        \mb E\bigl[U^\lambda+(1-U)^\lambda\bigr]
        =
        \frac{2\Gamma(\lambda+1/2)}
             {\sqrt\pi\,\Gamma(\lambda+1)} .
\]

Fix \(T<\infty\), and set
\[
        b_n(T):=\floor{{nT}/{2}} .
\]
For this fixed value of \(T\), we call any consecutive sequence of
\(b_n(T)\) updates a block.  We now describe the comparison used inside
one block.

Fix a starting time \(s\), a coordinate \(r\), and a mass
\(z\le Y_{s,r}\).  The construction uses the same future coordinate pairs
and beta variables as the energy chain.  At each time it consists of a
finite collection of coordinate-mass pairs \((\ell,m)\), interpreted as
a mass \(m\) placed at coordinate \(\ell\).  Initially,
\[
        \mc M_0=\{(r,z)\}.
\]

Suppose the selected pair is \(\{i,j\}\).  If neither \(i\) nor \(j\)
appears as a coordinate in the collection, the collection is unchanged.
If exactly one selected coordinate carries a mass \(m\), then replace
that mass by two masses placed at the selected coordinates, using the
same split as the energy-chain update.  Thus, if the energy-chain update
assigns the proportions \(U\) and \(1-U\) to \(i\) and \(j\),
respectively, the collection receives the masses \(Um\) at \(i\) and
\((1-U)m\) at \(j\).  If both selected coordinates carry masses from the
collection, declare a collision and stop updating the collection.  After
that time, the collection is kept equal to its value just before the
collision.  (Note that this does not happen with high probability.)

Before the collision time, no coordinate carries more than one mass from
the collection.  Moreover, for every update before this time,
\begin{equation}\label{eq:mass-domination-invariant}
        (\ell,m)\in\mc M_h \implies
        Y_{s+h,\ell}\ge m.
\end{equation}
This is true at time \(s\) by assumption.  For the inductive step, if a
selected coordinate carries a mass \(m\) from the collection, then \(m\)
is at most the total actual energy in the selected pair.  The actual
energy and this mass are then split in the same proportions \(U\) and
\(1-U\), so the inequality is preserved.  Coordinates outside the
selected pair are unchanged.

Run the construction through the block starting at time \(s\), using the
stopping convention above.  Define
\[
        B_{n,T}
        :=
        \min_{(\ell,m)\in\mc M_{b_n(T)}}
        \left(-\log\frac{m}{z}\right),
\]
and let \(\mc C_{n,T}\) be the event that a collision occurs during the
block.  Conditioned on the past up to time \(s\), the law of the relative
loss \(B_{n,T}\) depends only on \(n\) and \(T\), not on \(s,r,z\).

\begin{lemma}\label{lem:block-comparison}
For every fixed \(T<\infty\),
\[
        \mb E B_{n,T}\longrightarrow \mb E B_T,
        \qquad
        \sup_n \mb E B_{n,T}^2<\infty .
\]
Moreover,
\[
        \mb P(\mc C_{n,T})=O_T(n^{-1}).
\]
Finally, on \(\mc C_{n,T}^c\), if the construction is started from a mass
\(z\) dominated by the actual coordinate energy, then at the end of the
block
\begin{equation}\label{eq:mass-domination-output}
        \max_iY_{s+b_n(T),i}\ge z e^{-B_{n,T}} .
\end{equation}
\end{lemma}

We first deduce \cref{prop:persistence} from the block estimate.

\begin{proof}[Proof of \cref{prop:persistence} assuming \cref{lem:block-comparison}]
Fix \(0<c<C_{\mr{BRW}}=(2\gamma_{\mr{BRW}})^{-1}\).  Choose
\(\eta>0\) such that
\begin{equation}\label{eq:eta-choice-lower}
        2c(\gamma_{\mr{BRW}}+4\eta)<1 .
\end{equation}
By the \(L^1\)-convergence in \eqref{eq:brw-speed}, we may choose
\(T<\infty\) such that
\[
        \frac{\mb E B_{T'}}{T'}<\gamma_{\mr{BRW}}+\eta .
\]
Then, by \cref{lem:block-comparison}, for all sufficiently large
\(n\),
\[
        \frac{\mb E B_{n,T}}{T}<\gamma_{\mr{BRW}}+2\eta .
\]

Let
\[
        b_n:=b_n(T)=\floor{{nT}/{2}},
        \qquad
        q_n:=\floor{{t_n}/{b_n}} - 1,
        \qquad
        t_n=\ceil{c n\log n}.
\]
Then
\[
        q_nT=2c\log n+O(1).
\]

Break the first \(t_n\) updates into \(q_n\) consecutive blocks of
length \(b_n\) and one final block of size $t_n - q_n b_n$.  Note that the length of the final block is between $b_n$ and $2b_n$, and we will see this means it can be treated on equal footing to all other blocks.  Start with mass \(1\) at coordinate \(1\).  At the end
of each collision-free block, choose a coordinate-mass pair $(j,m)$ with the maximum $m$.
Use this pair as the initial coordinate and mass for the next
block.

The collision estimate in \cref{lem:block-comparison} applies at
the start of each block, conditioned on the past.  If
\(\mc F_{rb_n}\) denotes the past up to the beginning of block \(r+1\),
then, on the event that no collision has occurred in the first \(r\)
blocks,
\[
        \mb P\left(
        \text{a collision occurs in block }r+1
        \,\middle|\,\mc F_{rb_n}
        \right)
        \le
        C_T n^{-1}.
\]
Similarly, we may apply \cref{lem:block-comparison} to the final block to bound the probability at $C_{2T} n^{-1}$.  Therefore, by a union bound over the \(q_n+1\) blocks,
\[
        \mb P(\text{some collision occurs in any of the }q_n+1\text{ blocks})
        \le
        C_T q_n n^{-1} + C_{2T} n^{-1}
        =
        o(1),
\]
since \(q_n=O(\log n)\).  

On the no-collision event, \eqref{eq:mass-domination-output} applied
successively gives
\[
        L_{q_nb_n}
        \ge
        \exp\left\{-\sum_{r=1}^{q_n}B_{n,T}^{(r)} - B_{n,2T} \right\}.
\]
The random variables
\(B_{n,T}^{(1)},\ldots,B_{n,T}^{(q_n)}\) may be taken to be independent
copies of \(B_{n,T}\).  Indeed, conditioned on the past at the start of a
block, the relative loss over the next \(b_n\) updates has the same
law as \(B_{n,T}\), and different blocks use disjoint sets of future
updates.  The final block has length at most $2T$ and its logarithmic loss is therefore stochastically dominated by the loss of $B_{n,2T}$.  The second-moment bound in
\cref{lem:block-comparison} and Chebyshev's inequality imply
\[
        \sum_{r=1}^{q_n}B_{n,T}^{(r)} + B_{n,2T}
        \le
        (q_n+1)T(\gamma_{\mr{BRW}}+3\eta)
\]
with probability tending to one.



Combining the full blocks with the remaining updates, the total logarithmic loss by time \(t_n\) is at most
\[
        (q_n+1)T(\gamma_{\mr{BRW}}+3\eta)
        =
        2c(\gamma_{\mr{BRW}}+4\eta)\log n+O(1)
\]
with probability tending to one.  Hence
\[
        L_{t_n}
        \ge
        n^{-2c(\gamma_{\mr{BRW}}+4\eta)+o(1)}
\]
with probability tending to one.  Choose \(\beta\in(0,1)\) such that
\[
        2c(\gamma_{\mr{BRW}}+4\eta)<\beta<1,
\]
which is possible by \eqref{eq:eta-choice-lower}.  Then
\[
        \mb P_{e^{(1)}}(L_{t_n}\ge n^{-\beta})\longrightarrow1. \qedhere
\]
\end{proof}

It remains to prove the block estimate.

\begin{proof}[Proof of \cref{lem:block-comparison}]
Since the law of \(B_{n,T}\) depends only on \(n\) and \(T\), it is
enough to consider the construction started from mass \(1\) at coordinate
\(1\).  Let \(N_h:=|\mc M_h|\), with the convention that \(\mc M_h\) is
kept fixed after a collision.

For \(1\le k\le n\), set
\[
        p_{n,k}:=\frac{k(n-k)}{\binom n2},
        \qquad
        q_{n,k}:=\frac{\binom{k}{2}}{\binom n2}.
\]
If no collision has occurred and \(N_h=k\), then \(p_{n,k}\) is the
probability that the next selected pair contains exactly one coordinate
carrying a mass, and \(q_{n,k}\) is the probability that it contains two
such coordinates.

\paragraph{\bf Count and collision bounds.}
Let \(\mc F_h\) denote the sigma-algebra generated by the updates up to
time \(h\).  Since a split increases \(N_h\) by one and all other
outcomes do not increase it,
\begin{align*}
    \mb{E}[N_{h+1}^2 \mid \mc{F}_h] & \leq N_h^2 + \frac{N_h(n-N_h)}{\binom n2} \paren{ (N_h+1)^2 - N_h^2 } \\
    & \leq N_h^2 + \frac{2N_h}{n} (2N_h + 1) \leq \paren{1+\frac6n} N_h^2.
\end{align*}

Consequently, uniformly for
\(h\le b_n(T)=\floor{nT/2}\),
\begin{equation}\label{eq:count-second-moment}
        \mb E N_h^2\le e^{3T}.
\end{equation}

Using the definition of \(q_{n,k}\) and \eqref{eq:count-second-moment},
\[
        \mb P(\mc C_{n,T})
        \le
        \sum_{h=0}^{b_n(T)-1}
        \mb E\left[\frac{N_h^2}{n(n-1)}\right]
        =
        O_T(n^{-1}).
\]

\paragraph{\bf Approximation before the \(K\)-th split.}
Fix \(K\ge2\).  Stop the construction once \(N_h=K\), and let
\(B_{n,T}^{[K]}\) be the minimum logarithmic loss at time \(b_n(T)\) for
this stopped construction.  Define \(B_T^{[K]}\) in the same way for the
branching random walk, stopped when it first has \(K\) particles.  We
claim that
\begin{equation}\label{eq:fixed-k-convergence}
        B_{n,T}^{[K]}\Rightarrow B_T^{[K]} .
\end{equation}

We describe the first \(K-1\) splits.  Label the masses after each split
in any deterministic way.  If the construction currently has \(k<K\)
masses, then
\[
        p_{n,k}
        =
        \frac{2k}{n}+O_K(n^{-2}),
        \qquad
        q_{n,k}
        =
        O_K(n^{-2}).
\]
Thus the number of updates until the next split-or-collision event is
geometric with parameter
\[
        p_{n,k}+q_{n,k}=\frac{2k}{n}+O_K(n^{-2}).
\]
After multiplication by \(2/n\), this waiting time converges to an
exponential random variable of rate \(k\).  The next event is a collision
with probability
\[
        \frac{q_{n,k}}{p_{n,k}+q_{n,k}}=O_K(n^{-1}),
\]
so the probability that a collision occurs before the construction
reaches \(K\) masses tends to zero.

Given that the next event is a split while there are \(k\) masses, the
splitting mass is uniform among the \(k\) masses, because each of their
coordinates has exactly \(n-k\) possible partners not carrying a mass.
The beta variable used at the split is independent of the selected pair
and has law \(\on{Beta}(1/2,1/2)\).  Therefore, if \(W_{n,k}\) is the
number of updates between the \((k-1)\)-st and \(k\)-th splits,
\(R_{n,k}\in[k]\) is the index of the mass that splits, and \(U_{n,k}\)
is the corresponding beta variable, then
\[
\begin{aligned}
&\left(
        \frac{2W_{n,1}}n,\ldots,\frac{2W_{n,K-1}}n,\,
        R_{n,1},\ldots,R_{n,K-1},\,
        U_{n,1},\ldots,U_{n,K-1}
\right) \\
&\qquad\Rightarrow
\left(
        E_1,\ldots,E_{K-1},\,
        R_1,\ldots,R_{K-1},\,
        U_1,\ldots,U_{K-1}
\right),
\end{aligned}
\]
where \(E_k\) is exponential of rate \(k\), \(R_k\) is uniform on
\([k]\), and \(U_k\sim\on{Beta}(1/2,1/2)\), with all these variables
independent except for the deterministic ranges of the \(R_k\)'s.  This
is exactly the corresponding description of the branching random walk
before it reaches \(K\) particles.

The preceding variables determine the logarithmic losses present at time
\(T\), up to the stopping level \(K\): at the \(k\)-th split, the selected
loss \(x\) is replaced by
\[
        x-\log U_k,
        \qquad
        x-\log(1-U_k),
\]
and only splits whose accumulated rescaled time is at most \(T\) are
used.  This deterministic map is continuous except when a split time is
exactly \(T\).  The limiting process has no split at exactly time \(T\)
almost surely, so the continuous mapping theorem proves
\eqref{eq:fixed-k-convergence}.

\paragraph{\bf Removing the restriction to \(K\) masses.}
Since \(N_h\) is nondecreasing until the construction is stopped by a
collision, \eqref{eq:count-second-moment} gives
\[
        \sup_n
        \mb P\left(\max_{h\le b_n(T)}N_h\ge K\right)
        \le
        \frac{e^{3T}}{K^2}.
\]
For the branching random walk, let \(Z_T:=|\mc A_T|\).  Since each
particle branches at rate \(1\) and each branch increases the number of
particles by one,
\[
        \frac{d}{dT}\mb E Z_T=\mb E Z_T,
        \qquad
        \mb E Z_0=1.
\]
Hence \(\mb E Z_T=e^T\), and Markov's inequality gives
\[
        \mb P(|\mc A_T|\ge K)\le \frac{e^T}{K}
        \longrightarrow0
        \qquad\text{as }K\to\infty .
\]
Thus, for every bounded continuous \(f\),
\[
\begin{aligned}
 \limsup_{n\to\infty}
 \left|\mb E f(B_{n,T})-\mb E f(B_T)\right|
 &\le
 \limsup_{n\to\infty}
 \left|\mb E f(B_{n,T}^{[K]})-\mb E f(B_T^{[K]})\right|\\
 &\quad+
 2\norm{f}_\infty
 \limsup_{n\to\infty}
 \mb P\left(\max_{h\le b_n(T)}N_h\ge K\right)\\
 &\quad+
 2\norm{f}_\infty
 \mb P(|\mc A_T|\ge K).
\end{aligned}
\]
The first term is zero for fixed \(K\) by
\eqref{eq:fixed-k-convergence}.  Letting \(K\to\infty\)
proves
\[
        B_{n,T}\Rightarrow B_T .
\]

\paragraph{\bf Uniform integrability.}
During the \(b_n(T)\) updates, keep a single mass as follows.  Start with
the initial mass.  Whenever its current coordinate is selected and the
construction has not yet stopped, split this mass according to the update
and retain the larger of the two pieces.  Conditioned on the past and on
the current coordinate, the probability that the next selected pair
contains this coordinate is exactly \(2/n\).  Hence, by a sequential
coupling, the number \(H_n\) of such updates is stochastically dominated
by
\[
        \on{Bin}(b_n(T),2/n).
\]
At each such update, the retained mass is at least one half of its
previous value, so its logarithmic loss increases by at most \(\log2\).
Since \(B_{n,T}\) is the minimum logarithmic loss among the masses in
\(\mc M_{b_n(T)}\),
\[
        B_{n,T}\le(\log2)H_n .
\]
The binomial variables on the right have uniformly bounded second
moments, because \(b_n(T)\le nT/2\).  Therefore
\[
        \sup_n\mb E B_{n,T}^2<\infty .
\]
Together with \(B_{n,T}\Rightarrow B_T\), this implies
\[
        \mb E B_{n,T}\to \mb E B_T .
\]

Finally, on \(\mc C_{n,T}^c\), the domination invariant
\eqref{eq:mass-domination-invariant} holds through the end of the
block.  Applying it to a coordinate-mass pair
\((\ell,m)\in\mc M_{b_n(T)}\) with minimal logarithmic loss gives
\eqref{eq:mass-domination-output}.
\end{proof}

\section{An \(\ell^\infty\) estimate for the normalized scale vector}
\label{sec:flattening}

We use the scale-vector representation from \cref{prop:hidden-rep}.  Since
\(F_{\lambda a}=F_a\), we normalize \(A_t\) by its average coordinate.  Write
\[
        R_t:=\sum_iA_{t,i},
        \qquad
        \bar A_t:=\frac{R_t}{n},
        \qquad
        \xi_t:=A_t-\bar A_t\one,
        \qquad
        W_t:=\sum_i\xi_{t,i}^2 .
\]
Since \(A_t\in[0,\infty)^n\setminus\{0\}\), we have \(R_t>0\).
Define
\[
        u_t:=\frac{A_t}{\bar A_t}-\one
        =
        \frac{n\xi_t}{R_t}.
\]
Then \(A_t=\bar A_t(\one+u_t)\), and hence
\[
        F_{A_t}=F_{\one+u_t}.
\]

The result needed for the upper bound is the following.

\begin{proposition}
\label{prop:Linfty-full}
Let \(c>C_{\mr{BRW}}\), and let \(t_n=\ceil{c n\log n}\).  For every
fixed \(r>0\) and \(\varepsilon>0\),
\begin{equation}\label{eq:Linfty-full}
        \mb P\left(
        R_{t_n}\ge r,\ 
        \norm{u_{t_n}}_\infty>\varepsilon
        \right)
        \longrightarrow0 .
\end{equation}
\end{proposition}

The proof uses three estimates.  First, we control the lower tail of
\(R_t\), so that \(u_t=n\xi_t/R_t\) can be estimated on the event
\(\{R_t\ge r\}\).  Second, we prove a drift estimate for
\(\Phi_p(\xi_t)=\sum_i|\xi_{t,i}|^p\), with \(2<p<4\).  Third, we bound
the accumulated \(W_t^{p/2}\)-term in this drift using a fourth-moment
estimate for \(\xi_t\).  These estimates are combined at the end of the
section, where the choice of \(p\) gives the threshold \(C_{\mr{BRW}}\).

\subsection{The total scale}

We first record the quadratic contraction of \(\xi_t\) and the lower-tail
estimate for \(R_t\).

\begin{lemma}\label{lem:R-tail-full}
Let
\begin{equation*}
        \rho_n:=1-\frac{n+2}{2n(n-1)} .
\end{equation*}
Then
\begin{equation}\label{eq:EW-full}
        \mb E W_t=\rho_n^t\left(1-\frac1n\right).
\end{equation}
The process \(R_t\) is a martingale, converges in \(L^2\) to a limit
\(R_\infty\), and satisfies
\begin{equation*}
        \mb E[(R_\infty-R_t)^2]\le2\rho_n^t .
\end{equation*}
Moreover, for every sequence \(t_n\) with \(t_n/n\to\infty\), there is a
universal constant \(C<\infty\) such that
\begin{equation}\label{eq:R-t-tail-full}
        \limsup_{n\to\infty}\mb P(R_{t_n}\le r)
        \le C\sqrt r,
        \qquad 0<r\le1 .
\end{equation}
\end{lemma}

\begin{proof}
Suppose first that the selected pair is \(\{i,j\}\).  Then
\[
        A_{t+1,i}=A_{t+1,j}
        =
        BA_{t,i}+(1-B)A_{t,j},
\]
while all other coordinates are unchanged.  Hence
\[
\begin{aligned}
        R_{t+1}-R_t
        &=
        2\bigl(BA_{t,i}+(1-B)A_{t,j}\bigr)
        -A_{t,i}-A_{t,j}                                      \\
        &=
        (2B-1)(A_{t,i}-A_{t,j}).
\end{aligned}
\]
Since \(\mb E(2B-1)=0\), the process \(R_t\) is a martingale.

A direct calculation, conditioned on \(A_t\) and on the selected pair,
gives
\[
        \mb E[W_{t+1}\mid A_t,\{i,j\}]
        =
        W_t-\left(\frac14+\frac1{2n}\right)(A_{t,i}-A_{t,j})^2 .
\]
Averaging over the uniformly chosen pair and using
\[
        \sum_{i<j}(A_{t,i}-A_{t,j})^2=nW_t
\]
proves
\[
        \mb E W_{t+1}
        =
        \left(1-\frac{n+2}{2n(n-1)}\right)\mb E W_t .
\]
Since \(W_0=1-1/n\), this gives \eqref{eq:EW-full}.

Similarly,
\[
\begin{aligned}
        \mb E[(R_{t+1}-R_t)^2\mid A_t,\{i,j\}]
        &=
        \mb E(2B-1)^2\,(A_{t,i}-A_{t,j})^2                  \\
        &=
        \frac12(A_{t,i}-A_{t,j})^2 .
\end{aligned}
\]
Averaging over pairs gives
\[
        \mb E[(R_{t+1}-R_t)^2\mid A_t]
        =
        \frac{W_t}{n-1}.
\]
Therefore
\[
        \sum_{t=0}^{\infty}
        \mb E[(R_{t+1}-R_t)^2]
        =
        \frac1{n-1}\sum_{t=0}^{\infty}\mb E W_t
        <\infty .
\]
Thus \(R_t\) converges in \(L^2\) to a limit \(R_\infty\).  Orthogonality
of martingale increments and \eqref{eq:EW-full} give
\[
        \mb E[(R_\infty-R_t)^2]
        =
        \sum_{s=t}^{\infty}
        \mb E[(R_{s+1}-R_s)^2]
        =
        \frac1{n-1}\sum_{s=t}^{\infty}\mb E W_s
        \le2\rho_n^t .
\]

It remains to identify the lower tail of \(R_\infty\).  Write one scale
update as \(A'=HA\), where \(H\) is the random matrix determined by the
selected pair and beta variable in \eqref{eq:hidden-update}.  Then
\(H^{\mathsf T}\) acts on the simplex by
\[
        z_i'=B(z_i+z_j),
        \qquad
        z_j'=(1-B)(z_i+z_j),
\]
with all other coordinates unchanged.  Thus \(H^{\mathsf T}\) performs
one step of the energy chain.  Let \(\bar e=n^{-1}\one\).  Since
\(A_t=H_t\cdots H_1e^{(1)}\),
\[
        \frac{R_t}{n}
        =
        \ip{A_t}{\bar e}
        =
        \ip{e^{(1)}}{H_1^{\mathsf T}\cdots H_t^{\mathsf T}\bar e}.
\]
The matrices \(H_1,\ldots,H_t\) are iid, so
\(H_1^{\mathsf T}\cdots H_t^{\mathsf T}\) has the same law as
\(H_t^{\mathsf T}\cdots H_1^{\mathsf T}\), the product defining the
energy chain.  Hence
\[
        \frac{R_t}{n}\stackrel d=Y_{t,1},
\]
where \(Y_t\) is the energy chain started from \(\bar e\).
As noted in \cref{sub:squared-chain}, for fixed \(n\) this chain is
ergodic with stationary law
\[
        \pi_n=\on{Dir}_n(1/2,\ldots,1/2).
\]
Since \(R_t\to R_\infty\) in \(L^2\), if
\(Y^{\mr{stat}}\sim\pi_n\), then
\[
        \frac{R_\infty}{n}\stackrel d=Y^{\mr{stat}}_1 .
\]
Thus
\[
        R_\infty/n\sim\on{Beta}(1/2,(n-1)/2).
\]
The beta density gives, uniformly in \(n\),
\begin{align*}
        \mb P(R_\infty\le r) &
        =
        \paren{ \int_0^{1}u^{-1/2}(1-u)^{(n-3)/2}\,du }^{-1}
        \int_0^{r/n}u^{-1/2}(1-u)^{(n-3)/2}\,du \\
        & \le C\sqrt r,
        \qquad 0<r\le1 .
\end{align*}
Finally,
\[
        \mb P(R_t\le r)
        \le
        \mb P(R_\infty\le2r)+\mb P(|R_t-R_\infty|>r)
        \le
        C\sqrt r+\frac{2\rho_n^t}{r^2}.
\]
If \(t_n/n\to\infty\), then \(\rho_n^{t_n}\to0\), proving
\eqref{eq:R-t-tail-full}.
\end{proof}

\subsection{A drift estimate for \(\Phi_p\)}

Recall from \eqref{eq:m-lambda-main} that, for positive arguments and
\(U\sim\on{Beta}(1/2,1/2)\),
\[
        m(p)=
        \mb E[U^p+(1-U)^p]
        =
        \frac{2\Gamma(p+1/2)}
             {\sqrt\pi\,\Gamma(p+1)} .
\]
For \(2<p<4\), define
\[
        \Phi_p(x):=\sum_{i=1}^n |x_i|^p .
\]
The next proposition gives the drift estimate for \(\Phi_p(\xi_t)\).  Its
error term will be controlled in \cref{lem:W-moments-full}.

\begin{proposition}\label{prop:fractional-drift-full}
Fix \(2<p<4\).  There are constants \(C_p<\infty\) such that, for sufficiently large \(n\ge n_0(p)\) and every \(t\ge0\),
\[
        \mb E[\Phi_p(\xi_{t+1})\mid A_t]
        \le
        \left(1-\frac{2(1-m(p))}{n}\right)
        \Phi_p(\xi_t)
        +
        C_p n^{-p/2}W_t^{p/2}.
\]
\end{proposition}

We first isolate the two-variable estimate used in the proof.  The
restriction \(p<4\) is used here and again in the interpolation step in
\cref{lem:W-moments-full}.

\begin{lemma}\label{lem:two-variable}
Fix \(2<p<4\), let \(B\sim\on{Beta}(1/2,1/2)\), and set
\[
        \varphi_p(x):=|x|^{p-2}x,
        \qquad
        \ell_p:=p\,\mb E[B^{p-1}(1-B)] .
\]
There exists \(K_p<\infty\) such that, for all \(a,b\in\mb R\),
\begin{equation*}
        \mb E|Ba+(1-B)b|^p
        \le
        \frac{m(p)}2(|a|^p+|b|^p)
        +
        \ell_p\{\varphi_p(a)b+\varphi_p(b)a\}
        +
        K_p|ab|^{p/2}.
\end{equation*}
\end{lemma}

\begin{proof}
Define
\[
        \Delta_p(a,b)
        :=
        \mb E|Ba+(1-B)b|^p
        -
        \frac{m(p)}2(|a|^p+|b|^p) -
        \ell_p\{\varphi_p(a)b+\varphi_p(b)a\}.
\]
The function \(\Delta_p\) is homogeneous of degree \(p\).  It is
therefore enough to prove
\[
        \Delta_p(a,b)\le K_p|ab|^{p/2}
\]
on the unit circle \(|a|^2+|b|^2=1\).  Away from the two coordinate
axes, this follows by compactness.

It remains to treat the part of the unit circle where one coordinate is
small.  By symmetry, assume \(|b|\le |a|\), so that
\(|a|\ge2^{-1/2}\).  Taylor's inequality,
applied to \(Ba+(1-B)b\) around \(Ba\), gives
\[
\begin{aligned}
        |Ba+(1-B)b|^p
        &\le
        |Ba|^p
        +
        p|Ba|^{p-2}(Ba)(1-B)b                              \\
        &\quad
        +
        C_p\left(
        |Ba|^{p-2}(1-B)^2b^2+|(1-B)b|^p
        \right).
\end{aligned}
\]
Averaging over \(B\) yields
\[
        \mb E|Ba+(1-B)b|^p
        \le
        \frac{m(p)}2|a|^p
        +
        \ell_p\varphi_p(a)b
        +
        C_p\bigl(|a|^{p-2}b^2+|b|^p\bigr),
\]
where we used
\[
        \mb E[B^{p-2}(1-B)^2]<\infty,
\]
which holds because \(p>2\).  Comparing this bound with the definition
of \(\Delta_p(a,b)\), we obtain
\[
\begin{aligned}
        \Delta_p(a,b)
        &\le
        C_p\bigl(|a|^{p-2}b^2+|b|^p\bigr)
        -
        \frac{m(p)}2|b|^p
        -
        \ell_p\varphi_p(b)a                                      \\
        &\le
        C_p\bigl(|a|^{p-2}b^2+|b|^p+|a||b|^{p-1}\bigr).
\end{aligned}
\]
The last inequality uses
\[
        \ell_p|\varphi_p(b)a|
        =
        \ell_p|a||b|^{p-1}.
\]

Since \(|a|\asymp1\), \(|b|\le1\), and \(2<p<4\), each of the three
terms satisfies
\[
        |a|^{p-2}b^2,\quad |b|^p,\quad |a||b|^{p-1}
        \le
        C_p|b|^{p/2}.
\]
On the present part of the unit circle, \(|a|\asymp1\), so
\[
        |b|^{p/2}\asymp |ab|^{p/2}.
\]
Hence
\[
        \Delta_p(a,b)\le K_p|ab|^{p/2}
\]
near the axis \(b=0\).  The neighborhood of the other axis is identical.
\end{proof}

\begin{proof}[Proof of \cref{prop:fractional-drift-full}]
Work conditioned on \(A_t\), and write
\[
        \xi_i:=\xi_{t,i},
        \qquad
        W_t=\sum_i\xi_i^2 .
\]
Suppose first that the selected pair is \(\{i,j\}\).  Since
\[
        A_{t,k}=\bar A_t+\xi_k ,
\]
the scale update gives
\[
        A_{t+1,i}=A_{t+1,j}
        =
        B(\bar A_t+\xi_i)+(1-B)(\bar A_t+\xi_j)
        =
        \bar A_t+B\xi_i+(1-B)\xi_j,
\]
and
\[
        A_{t+1,k}=\bar A_t+\xi_k,
        \qquad k\notin\{i,j\}.
\]
Thus
\[
        A_{t+1}=\bar A_t\one+z,
\]
where
\[
        z_i=z_j=B\xi_i+(1-B)\xi_j,
        \qquad
        z_k=\xi_k\quad(k\notin\{i,j\}).
\]
Since \(\sum_k\xi_k=0\),
\[
        \frac1n\sum_k z_k
        =
        \frac{(2B-1)(\xi_i-\xi_j)}{n}.
\]
Therefore
\[
        \bar A_{t+1}
        =
        \bar A_t+h,
        \qquad
        h:=\frac{(2B-1)(\xi_i-\xi_j)}{n},
\]
and hence
\[
        \xi_{t+1}
        =
        A_{t+1}-\bar A_{t+1}\one
        =
        z-h\one .
\]
We first estimate \(\Phi_p(z)\), and then account for the centering by
\(h\one\).

For the fixed selected pair \(\{i,j\}\),
\[
        \Phi_p(z)
        =
        \Phi_p(\xi_t)
        -|\xi_i|^p-|\xi_j|^p
        +
        2|B\xi_i+(1-B)\xi_j|^p .
\]
Averaging over the uniformly selected pair gives
\[
\begin{aligned}
        \mb E[\Phi_p(z)\mid A_t]
        &=
        \left(1-\frac2n\right)\Phi_p(\xi_t)
        +
        \frac{4}{n(n-1)}
        \sum_{i<j}
        \mb E|B\xi_i+(1-B)\xi_j|^p .
\end{aligned}
\]
We now apply \cref{lem:two-variable} to each term in the last sum.  The
terms involving \(m(p)\) contribute
\[
        \frac{4}{n(n-1)}
        \sum_{i<j}
        \frac{m(p)}2\bigl(|\xi_i|^p+|\xi_j|^p\bigr)
        =
        \frac{2m(p)}n\Phi_p(\xi_t).
\]
The terms involving \(\ell_p\) give
\[
        \frac{4\ell_p}{n(n-1)}
        \sum_{i<j}
        \{\varphi_p(\xi_i)\xi_j+\varphi_p(\xi_j)\xi_i\}.
\]
Since \(\sum_i\xi_i=0\),
\[
        \sum_{i<j}
        \{\varphi_p(\xi_i)\xi_j+\varphi_p(\xi_j)\xi_i\}
        =
        \sum_{i\ne j}\varphi_p(\xi_i)\xi_j                 
        =
        \sum_i\varphi_p(\xi_i)\sum_{j\ne i}\xi_j
        =
        -\Phi_p(\xi_t).
\]
Thus these terms contribute
\[
        -\frac{4\ell_p}{n(n-1)}\Phi_p(\xi_t).
\]
The remaining term in \cref{lem:two-variable} is
\[
        \frac{4K_p}{n(n-1)}
        \sum_{i<j}|\xi_i\xi_j|^{p/2}.
\]
Combining these estimates,
\[
\begin{aligned}
        \mb E[\Phi_p(z)\mid A_t]
        &\le
        \left(
        1-\frac2n+\frac{2m(p)}n
        -\frac{4\ell_p}{n(n-1)}
        \right)\Phi_p(\xi_t)                                  \\
        &\quad
        +
        \frac{4K_p}{n(n-1)}
        \sum_{i<j}|\xi_i\xi_j|^{p/2}.
\end{aligned}
\]
Using
\[
        \sum_{i<j}|\xi_i\xi_j|^{p/2}
        \le
        \frac12\left(\sum_i|\xi_i|^{p/2}\right)^2
        \le
        \frac12 n^{2-p/2}W_t^{p/2},
\]
we obtain
\begin{equation}\label{eq:raw-drift-full}
        \mb E[\Phi_p(z)\mid A_t]
        \le
        \left(1-\frac{2(1-m(p))}{n}-\frac{4\ell_p}{n(n-1)}\right)\Phi_p(\xi_t)
        +
        C_p n^{-p/2}W_t^{p/2}.
\end{equation}

It remains to compare \(\Phi_p(z-h\one)\) with \(\Phi_p(z)\).  Taylor's
inequality gives
\begin{equation}\label{eq:recentering-taylor-full}
        \Phi_p(z-h\one)
        \le
        \Phi_p(z)
        -
        ph\sum_k\varphi_p(z_k)
        +
        C_p\left(
        h^2\sum_k|z_k|^{p-2}+n|h|^p
        \right).
\end{equation}

We first show that the linear term in \eqref{eq:recentering-taylor-full}
has nonpositive expectation after conditioning on \(A_t\) and on the
selected pair.  The coordinates outside \(\{i,j\}\) contribute zero to
\[
        \mb E\left[
        h\sum_k\varphi_p(z_k)
        \,\middle|\, A_t,\{i,j\}
        \right],
\]
because they do not depend on \(B\), while \(\mb E h=0\).  The two
selected coordinates contribute
\[
        \frac{2(\xi_i-\xi_j)}{n}
        \mb E\left[
        (2B-1)\varphi_p(B\xi_i+(1-B)\xi_j)
        \right].
\]
This quantity is nonnegative.  Indeed, using the symmetry
\(B\stackrel d=1-B\), the expectation in the last display equals
\[
        \frac12
        \mb E\left[
        (2B-1)
        \left\{
        \varphi_p(B\xi_i+(1-B)\xi_j)
        -
        \varphi_p((1-B)\xi_i+B\xi_j)
        \right\}
        \right].
\]
Since \(\varphi_p\) is increasing, the difference inside braces has the
same sign as
\[
        (2B-1)(\xi_i-\xi_j).
\]
It follows that the last displayed expectation has the same sign as
\(\xi_i-\xi_j\).  After multiplication by the prefactor
\((\xi_i-\xi_j)\), the selected-coordinate contribution is nonnegative.
Therefore
\[
        \mb E\left[
        h\sum_k\varphi_p(z_k)
        \,\middle|\, A_t,\{i,j\}
        \right]\ge0,
\]
and hence
\[
        \mb E\left[
        -ph\sum_k\varphi_p(z_k)
        \,\middle|\, A_t,\{i,j\}
        \right]\le0 .
\]

It remains to bound the two remainder terms in
\eqref{eq:recentering-taylor-full}.  Since \(p-2\in(0,2)\), H\"older's inequality gives
\[
        \sum_k|z_k|^{p-2}
        \le
        2\sum_k|\xi_k|^{p-2}
        \le
        2 n^{(4-p)/2}W_t^{(p-2)/2}.
\]
Also,
\[
        \mb E_{\{i,j\}}(\xi_i-\xi_j)^2
        =
        \frac{2}{n-1}W_t .
\]
Using \(|2B-1|\le1\), we get
\[
\begin{aligned}
        \mb E\left[
        h^2\sum_k|z_k|^{p-2}
        \,\middle|\, A_t
        \right]
        &\le
        2 n^{-2}
        n^{(4-p)/2}W_t^{(p-2)/2}
        \mb E_{\{i,j\}}(\xi_i-\xi_j)^2                         \\
        &\le
        2 n^{-(p+2)/2}W_t^{p/2}
        \le
        2 n^{-p/2}W_t^{p/2}.
\end{aligned}
\]
Similarly,
\[
        \mb E_{\{i,j\}}|\xi_i-\xi_j|^p
        \le
        \frac{D_p}{n}\Phi_p(\xi_t),
\]
and therefore
\[
\begin{aligned}
        \mb E[n|h|^p\mid A_t]
        &\le
        D_p n^{1-p}
        \mb E_{\{i,j\}}|\xi_i-\xi_j|^p  \\
        &\le
        D_p n^{-p}\Phi_p(\xi_t).
\end{aligned}
\]

Combining these estimates with \eqref{eq:raw-drift-full} gives
\[
        \mb E[\Phi_p(\xi_{t+1})\mid A_t]
        \le
        \left(
        1-\frac{2(1-m(p))}{n}-\frac{4\ell_p}{n(n-1)}+\frac{D_p}{n^p}
        \right)
        \Phi_p(\xi_t)
        +
        C_p n^{-p/2}W_t^{p/2}.
\]

Since $p > 2$, choosing sufficiently large $n = n(p)$ ensures that the $\ell_p$ term is at least as large as the $D_p$ term, proving the proposition.
\end{proof}
\subsection{The \(W_t^{p/2}\) term}

The error term in \cref{prop:fractional-drift-full} is
\[
        n^{-p/2}W_t^{p/2}.
\]
We control it by combining the exact first-moment estimate
\[
        \mb E W_t
        =
        \rho_n^t\left(1-\frac1n\right),
        \qquad
        \rho_n=1-\frac1{2n}+O(n^{-2}),
\]
with a second-moment estimate for \(W_t\), equivalently a fourth-moment
estimate for the centered scale vector \(\xi_t\).

\begin{lemma}\label{lem:W-moments-full}
For every \(\varepsilon>0\), there are constants
\(n_0=n_0(\varepsilon)\) and \(C_\varepsilon<\infty\) such that, for all
\(n\ge n_0\) and all \(t\ge0\),
\begin{equation}\label{eq:EW2-soft}
        \mb E W_t^2
        \le
        C_\varepsilon
        \exp\left\{-\left(\frac{29}{32}-\varepsilon\right)\frac{t}{n}\right\}.
\end{equation}
Consequently, for every \(2<p<4\) and every \(\varepsilon>0\), with
\begin{equation}\label{eq:sigma-p-full}
        \sigma_p
        :=
        \left(2-\frac p2\right)\frac12
        +
        \left(\frac p2-1\right)\frac{29}{32}
        =
        \frac{6+13p}{64},
\end{equation}
there are constants \(n_0=n_0(p,\varepsilon)\) and
\(C_{p,\varepsilon}<\infty\) such that, for all \(n\ge n_0\) and all
\(t\ge0\),
\begin{equation}\label{eq:EWp2-soft}
        \mb E W_t^{p/2}
        \le
        C_{p,\varepsilon}
        \exp\left\{-\left(\sigma_p-\varepsilon\right)\frac{t}{n}\right\}.
\end{equation}
\end{lemma}

\begin{remark}
The properties of \(\sigma_p\) used below are verified in
\cref{lem:numerical-choice}: for \(p\) chosen sufficiently close to the
minimizer of \(p/[2(1-m(p))]\),
\[
        \sigma_p<2(1-m(p))
        \qquad\text{and}\qquad
        \frac{1+p/2}{\sigma_p}<C_{\mr{BRW}}.
\]
\end{remark}

\begin{proof}
For convenience of notation, let
\[
        M_t:=\mb E W_t^2,
        \qquad
        V_t:=\mb E\sum_i\xi_{t,i}^4 .
\]
By the calculation in \cref{app:fourth-moment},
\[
        \binom{M_{t+1}}{V_{t+1}}
        =
        \begin{pmatrix}
        1-\frac1n+O(n^{-2})
        &
        \frac{19}{16n}+O(n^{-2})\\[0.7ex]
        O(n^{-2})
        &
        1-\frac{29}{32n}+O(n^{-2})
        \end{pmatrix}
        \binom{M_t}{V_t},
\]
uniformly as \(n\to\infty\).  Fix \(\varepsilon>0\), and choose
\(L<\infty\) so large that
\[
        L\varepsilon>\frac{19}{16}.
\]

For all sufficiently large \(n\), expanding the matrix recursion gives
\[
\begin{aligned}
        M_{t+1}+LV_{t+1}
        &\le
        \left(1-\frac1n+O_L(n^{-2})\right)M_t       \\
        &\quad+
        \left[
        L\left(1-\frac{29}{32n}+O(n^{-2})\right)
        +
        \frac{19}{16n}+O(n^{-2})
        \right]V_t .
\end{aligned}
\]
The coefficient of \(M_t\) is at most
\[
        1-\left(\frac{29}{32}-\varepsilon\right)\frac1n
\]
for all large \(n\), since \(1>29/32-\varepsilon\).  For the coefficient
of \(V_t\), the choice \(L\varepsilon>19/16\) gives
\[
        L\left(1-\frac{29}{32n}+O(n^{-2})\right)
        +
        \frac{19}{16n}+O(n^{-2})
        \le
        L\left[
        1-\left(\frac{29}{32}-\varepsilon\right)\frac1n
        \right]
\]
for all large \(n\).  Hence
\[
        M_{t+1}+LV_{t+1}
        \le
        \left[
        1-\left(\frac{29}{32}-\varepsilon\right)\frac1n
        \right]
        (M_t+LV_t).
\]
Since \(M_0=O(1)\) and \(V_0=O(1)\), iteration gives
\[
        M_t
        \le
        M_t+LV_t
        \le
        C_\varepsilon
        \exp\left\{-\left(\frac{29}{32}-\varepsilon\right)\frac{t}{n}\right\}.
\]
This proves \eqref{eq:EW2-soft}.

We now interpolate.  Since \(2<p<4\), log-convexity gives
\[
        \mb E W_t^{p/2}
        \le
        (\mb E W_t)^{2-p/2}
        (\mb E W_t^2)^{p/2-1}.
\]
From \eqref{eq:EW-full} and
\[
        \rho_n=1-\frac1{2n}+O(n^{-2}),
\]
we have, for every \(\varepsilon>0\) and all sufficiently large \(n\),
\[
        \mb E W_t
        \le
        C_\varepsilon
        \exp\left\{-\left(\frac12-\varepsilon\right)\frac{t}{n}\right\}.
\]
Combining this bound with \eqref{eq:EW2-soft} gives
\[
\begin{aligned}
        \mb E W_t^{p/2}
        &\le
        C_{p,\varepsilon}
        \exp\left\{
        -\left[
        \left(2-\frac p2\right)\left(\frac12-\varepsilon\right)
        +
        \left(\frac p2-1\right)\left(\frac{29}{32}-\varepsilon\right)
        \right]\frac{t}{n}
        \right\} .
\end{aligned}
\]
Since
\[
        \left(2-\frac p2\right)+\left(\frac p2-1\right)=1,
\]
the exponent is exactly
\[
        -\left(\sigma_p-\varepsilon\right)\frac{t}{n}.
\]
Thus
\[
        \mb E W_t^{p/2}
        \le
        C_{p,\varepsilon}
        \exp\left\{-\left(\sigma_p-\varepsilon\right)\frac{t}{n}\right\},
\]
which proves \eqref{eq:EWp2-soft}.
\end{proof}

\subsection{Choice of \(p\) and proof of \cref{prop:Linfty-full}}

We now combine the previous estimates to prove \cref{prop:Linfty-full}.
This is the point in the proof of the upper bound where the constant
\(C_{\mr{BRW}}\) appears.

Fix \(2<p<4\).  The drift estimate in
\cref{prop:fractional-drift-full} shows that, apart from the accumulated
\(W_t^{p/2}\)-term,
\(\mb E\Phi_p(\xi_t)\) decays at rate
\[
        \frac{2(1-m(p))}{n}.
\]
Thus the contribution of the initial \(\Phi_p\)-mass at time
\(t=c n\log n\) is of order
\[
        n^{-2c(1-m(p))}.
\]
To convert this into an \(\ell^\infty\)-bound for
\[
        u_t=\frac{n\xi_t}{R_t},
\]
on the event \(\{R_t\ge r\}\), we use Markov's inequality:
\[
        \mb P\left(
        R_t\ge r,\ \norm{u_t}_\infty>\varepsilon
        \right)
        \le
        \frac{n^p}{r^p\varepsilon^p}\mb E\Phi_p(\xi_t).
\]
The contribution of the initial condition is therefore negligible at time
\(c n\log n\) provided
\[
        n^p n^{-2c(1-m(p))}\to0,
\]
or equivalently
\[
        c>\frac{p}{2(1-m(p))}.
\]
The numerical verification in \cref{app:numerical} shows that the
minimizer of \(p/[2(1-m(p))]\) lies in \((2,4)\).  Hence optimizing this
condition over the available range of \(p\)'s gives
\[
        \inf_{p>1}\frac{p}{2(1-m(p))}=C_{\mr{BRW}}.
\]

It remains to check that the accumulated \(W_t^{p/2}\)-term does not
force a larger value of \(c\).  The estimate in
\cref{lem:W-moments-full} controls this term through the exponent
\[
        \sigma_p:=\frac{6+13p}{64}.
\]
After the same Markov conversion, the error term is negligible provided
\[
        c>\frac{1+p/2}{\sigma_p}.
\]
We will also use that, for the chosen \(p\),
\[
        \sigma_p<2(1-m(p)),
\]
so that the convolution estimate below is governed by the slower decay
coming from the \(W_t^{p/2}\)-term.  The following numerical input
verifies these inequalities for some \(p\in(2,4)\) whenever
\(c>C_{\mr{BRW}}\).

\begin{lemma}
\label{lem:numerical-choice}
One has \(C_{\mr{BRW}}>3\).  Moreover, for every
\(c>C_{\mr{BRW}}\), there exists \(p\in(2,4)\) such that
\[
        c>\frac{p}{2(1-m(p))},
        \qquad
        c>\frac{1+p/2}{\sigma_p},
        \qquad
        2(1-m(p))>\sigma_p.
\]
\end{lemma}

The certified verification is given in \cref{app:numerical}.  The
inequality \(C_{\mr{BRW}}>3\) will be used only in \cref{sec:tv-upgrade}.

\begin{proof}[Proof of \cref{prop:Linfty-full}]
Choose \(p\in(2,4)\) as in \cref{lem:numerical-choice}.  Choose
\(\theta< \sigma_p \) sufficiently close to \(\sigma_p\) so that
\begin{equation}\label{eq:theta-choice-flattening}
        c\theta>1+\frac p2,
        \qquad
        \theta<2(1-m(p)).
\end{equation}
By \cref{prop:fractional-drift-full} and \cref{lem:W-moments-full}, for
all sufficiently large \(n\),
\[
\begin{aligned}
        \mb E\Phi_p(\xi_{t+1})
        &\le
        \left(
        1-\frac{2(1-m(p))}{n}
        \right)
        \mb E\Phi_p(\xi_t)
        +
        C_{p,\theta}n^{-p/2}e^{-\theta t/n}.
\end{aligned}
\]
Iterating the recursion from time \(0\) to time \(t\), and using
\[
        \Phi_p(e^{(1)}-n^{-1}\one)=O(1),
\]
gives
\begin{align*}
        \mb E\Phi_p(\xi_t)
        &\le
        C_{p,\theta}
        \exp\left\{
        -2(1-m(p))\frac{t}{n}
        \right\}
        +
        C_{p,\theta}n^{-p/2}
        \sum_{s=0}^{t-1}
        \exp\left\{
        -2(1-m(p))\frac{t-1-s}{n}
        \right\}
        e^{-\theta s/n} \\
        & \leq C_{p,\theta} \exp\left\{ -2(1-m(p)) \frac tn \right\} + C_{p,\theta} n^{-p/2} e^{-\theta t/n} \sum_{s=0}^{t-1} \exp\left\{ -\paren{2(1-m(p)) - \theta} \frac sn \right\} \\
        & \leq C_{p,\theta} \exp\left\{ -2(1-m(p)) \frac tn \right\} + C_{p,\theta} n^{1-p/2} e^{-\theta t/n}. \stepcounter{equation}\tag{\theequation} \label{eq:Phi-bound-soft}
\end{align*}

Now let \(t_n=\ceil{c n\log n}\).  On \(\{R_{t_n}\ge r\}\),
\[
        |u_{t_n,i}|
        =
        \frac{n|\xi_{t_n,i}|}{R_{t_n}}
        \le
        \frac{n|\xi_{t_n,i}|}{r}.
\]
Hence
\[
        \left\{
        R_{t_n}\ge r,\ \norm{u_{t_n}}_\infty>\varepsilon
        \right\}
        \subseteq
        \left\{
        \Phi_p(\xi_{t_n})
        >
        \left(\frac{r\varepsilon}{n}\right)^p
        \right\}.
\]
By Markov's inequality and \eqref{eq:Phi-bound-soft},
\[
\begin{aligned}
        \mb P\left(
        R_{t_n}\ge r,\ 
        \norm{u_{t_n}}_\infty>\varepsilon
        \right)
        &\le
        \frac{n^p}{r^p\varepsilon^p}
        \mb E\Phi_p(\xi_{t_n})                                      \\
        &\le
        C_{r,\varepsilon,p,\theta}(
        n^{p-c\cdot2(1-m(p))} + n^{1 + p/2 - c\theta})
        .
\end{aligned}
\]

Both exponents are negative by \cref{eq:theta-choice-flattening}.  This proves \eqref{eq:Linfty-full}.
\end{proof}

\section{From scale-vector estimates to total variation}
\label{sec:tv-upgrade}

We prove total-variation convergence by controlling \(\chi^2\)-divergence.
For probability measures \(\mu,\nu\) on the same measurable space, define
\[
        \chi^2(\mu\Vert\nu)
        :=
        \begin{cases}
        \displaystyle
        \int \left(\frac{d\mu}{d\nu}-1\right)^2\,d\nu,
        & \mu\ll\nu,\\[2ex]
        \infty, & \text{otherwise}.
        \end{cases}
\]
Thus,
\[
        \norm{\mu-\nu}_{\on{TV}}
        =
        \frac12\int\left|\frac{d\mu}{d\nu}-1\right|\,d\nu
        \le
        \frac12\sqrt{\chi^2(\mu\Vert\nu)}.
\]
For the mixtures \(\mb E[F_A]\) arising from the scale-vector
representation, the following lemma bounds the \(\chi^2\)-divergence
in terms of two independent copies of \(A\).

\begin{lemma}\label{lem:product-gamma-full}
Let \(A\) be a random vector in \((0,\infty)^n\), let
\[
        \mu:=\mb E[F_A],
\]
and let \(A'\) be an independent copy of \(A\).  Then
\begin{equation}\label{eq:chi-product-full}
        1+\chi^2(\mu\Vert\pi_n)
        \le
        \mb E_{A,A'}
        \prod_{i=1}^n(A_i+A_i'-A_iA_i')^{-1/2},
\end{equation}
with the convention that the right-hand side is infinite if one of the
one-dimensional integrals diverges.
\end{lemma}

\begin{proof}
For \(a=(a_1,\ldots,a_n)\in(0,\infty)^n\), let \(\Gamma_a\) be the
product law of independent \(\on{Gamma}(1/2,a_i)\) random variables.  The
measure \(F_a\) is the push-forward of \(\Gamma_a\) under the
normalization map
\[
        x\mapsto \frac{x}{\sum_i x_i}.
\]
Let
\[
        \bar\Gamma:=\mb E[\Gamma_A].
\]
The measures \(\mu\) and \(\pi_n=F_{\one}\) are the push-forwards of
\(\bar\Gamma\) and \(\Gamma_{\one}\), respectively, under the
normalization map.  By data processing for \(\chi^2\)-divergence,
\[
        \chi^2(\mu\Vert\pi_n)
        \le
        \chi^2(\bar\Gamma\Vert\Gamma_{\one}).
\]
Writing \(L_a=d\Gamma_a/d\Gamma_{\one}\), Fubini gives
\[
\begin{aligned}
        1+\chi^2(\bar\Gamma\Vert\Gamma_{\one})
        &=
        \int \left(\mb E_A L_A(x)\right)^2\,\Gamma_{\one}(dx)       \\
        &=
        \mb E_{A,A'}
        \int L_A(x)L_{A'}(x)\,\Gamma_{\one}(dx).
\end{aligned}
\]
The last integral factorizes over coordinates.  In one coordinate, for
shape \(1/2\) and scale parameters \(a,b>0\),
\[
        \int L_aL_b\,d\on{Gamma}(1/2,1)
        =
        (a+b-ab)^{-1/2}
\]
when \(a+b-ab>0\), and the integral is infinite otherwise.  Multiplying
over coordinates proves \eqref{eq:chi-product-full}.
\end{proof}

We next describe how the product in \eqref{eq:chi-product-full} will be
estimated.  Fix \(c>3\), and set
\[
        t_n=\ceil{c n\log n}.
\]
If \(u\) and \(v\) are independent copies of \(u_{t_n}\), then
\[
        (1+u_i)+(1+v_i)-(1+u_i)(1+v_i)
        =
        1-u_iv_i.
\]
Thus applying \cref{lem:product-gamma-full} with \(A=\one+u\) and
\(A'=\one+v\) leads to the product
\[
        \prod_{i=1}^n(1-u_iv_i)^{-1/2}.
\]

We will use the numerical inequality
\begin{equation}\label{eq:log-overlap-quadratic}
        -\frac12\log(1-x)
        \le
        \frac{x}{2}+\frac12x^2,
        \qquad |x|\le\frac12 .
\end{equation}
On the event where \(\norm{u}_\infty,\norm{v}_\infty\le\delta\), this gives
\[
        \log\prod_{i=1}^n(1-u_iv_i)^{-1/2}
        \le
        \frac12\sum_{i=1}^n u_iv_i
        +
        \frac12\sum_{i=1}^n u_i^2v_i^2 .
\]
The \(\ell^\infty\)-bound justifies this expansion, but it does not by
itself control the signed sum \(\sum_i u_iv_i\).  The cancellation comes
from the coordinates \(2,\ldots,n\).  The law of the scale process
started from \(e^{(1)}\) is invariant under permutations of these
coordinates.  Thus, after conditioning on \(u_1,v_1\) and on the two
multisets
\[
        \{u_2,\ldots,u_n\},
        \qquad
        \{v_2,\ldots,v_n\},
\]
the contribution from coordinates \(2,\ldots,n\) is a uniformly permuted
pairing.  We use the following estimate for such pairings.  It will be
applied with \(\lambda=D=1/2\).  The proof is deferred to
\cref{sub:permutation-estimate}.

\begin{lemma}\label{lem:permutation-full}
Fix \(\lambda\in\mb R\) and \(D\ge0\).  There exist constants
\(0<\delta_0<1\) and \(C<\infty\) such that the following holds.  Let
\(m\ge1\), let \(a,b\in\mb R^m\) satisfy
\[
        \max_i |a_i|\vee\max_i |b_i|\le\delta_0,
\]
and let \(\Pi\) be a uniform random permutation of \(\{1,\ldots,m\}\).
Then
\begin{equation}\label{eq:perm-bound-full}
 \mb E_\Pi\exp\left\{
 \lambda\sum_{i=1}^m a_i b_{\Pi(i)}
 +D\sum_{i=1}^m a_i^2b_{\Pi(i)}^2
 \right\}
 \le
 \exp\left\{
 \lambda\frac{(\sum_i a_i)(\sum_i b_i)}{m}
 +C\frac{\norm{a}_2^2\norm{b}_2^2}{m}
 \right\}.
\end{equation}
\end{lemma}

Let \(\delta_0\) be the constant in \cref{lem:permutation-full} with
\(\lambda=D=1/2\).  By decreasing \(\delta_0\), we assume
that \(\delta_0\le1/4\).

For \(r>0\), set, with \(\rho_n\) as in \cref{lem:R-tail-full},
\[
        L_n(r):=\frac{n^2\rho_n^{t_n}}{r^2}.
\]
For \(K<\infty\) and \(0<\delta<\delta_0\), define the good event
\begin{equation}\label{eq:good-event-full}
        G_n(r,K,\delta)
        :=
        \left\{
        R_{t_n}\ge r,\ 
        \norm{u_{t_n}}_\infty\le\delta,\ 
        \norm{u_{t_n}}_2^2\le K L_n(r)
        \right\}.
\end{equation}
The three conditions in the good event have separate roles.  The lower
bound on \(R_{t_n}\) lets us estimate \(u_{t_n}=n\xi_{t_n}/R_{t_n}\).
The \(\ell^\infty\)-bound justifies \eqref{eq:log-overlap-quadratic} and
allows us to apply \cref{lem:permutation-full}.  The \(\ell^2\)-bound
controls the term
\[
        \frac{\norm{u}_2^2\norm{v}_2^2}{n}
\]
in \eqref{eq:perm-bound-full}.  Indeed, on \(\{R_{t_n}\ge r\}\),
\[
        \norm{u_{t_n}}_2^2
        =
        \frac{n^2W_{t_n}}{R_{t_n}^2}
        \le
        \frac{n^2W_{t_n}}{r^2}.
\]
Since \(\mb E W_{t_n}=\rho_n^{t_n}(1-1/n)\), the quantity \(L_n(r)\) is
the corresponding first-moment scale for \(\norm{u_{t_n}}_2^2\) on this
event.  The event \(G_n(r,K,\delta)\) is invariant under permutations of
coordinates \(2,\ldots,n\).

\begin{proposition}
\label{prop:good-smoothing-full}
Let \(t_n=\ceil{c n\log n}\) with \(c>3\).  Fix \(r>0\),
\(K<\infty\), and \(0<\delta<\delta_0\).  Whenever
\(\mb P(G_n(r,K,\delta))>0\), set
\[
        \mu_{n,G}:=
        \mb E\left[F_{A_t}\mid G_n(r,K,\delta)\right].
\]
Then
\begin{equation}\label{eq:good-chi-full}
        \limsup_{n\to\infty}
        \chi^2(\mu_{n,G}\Vert\pi_n)
        \le
        e^{C\delta^2}-1,
\end{equation}
where \(C\) is universal.
\end{proposition}

\begin{proof}
Recall $A_t = \one + u_t$.  Let \(u,v\) be independent copies of \(u_{t_n}\), each conditioned on
the good event \(G_n(r,K,\delta)\).

\paragraph{\bf The \(\chi^2\) estimate.}
On the good event,
\[
        1+u_i\ge1-\delta>0
\]
for every \(i\), so \cref{lem:product-gamma-full} applies to
\(A=\one+u\).  It gives
\begin{equation}\label{eq:J-bound-full}
        1+\chi^2(\mu_{n,G}\Vert\pi_n)
        \le
        \mb E J(u,v),
        \qquad
        J(u,v):=\prod_{i=1}^n(1-u_iv_i)^{-1/2}.
\end{equation}
Since \(|u_iv_i|\le\delta^2<1/2\), \eqref{eq:log-overlap-quadratic}
implies
\begin{equation}\label{eq:J-log-full}
        \log J(u,v)
        \le
        \frac12\sum_{i=1}^n u_iv_i
        +
        \frac12\sum_{i=1}^n u_i^2v_i^2 .
\end{equation}

\paragraph{\bf The permutation estimate.}
The law of the scale process started from \(e^{(1)}\) is invariant under
permutations of coordinates \(2,\ldots,n\), and the good event
\(G_n(r,K,\delta)\) has the same invariance.  Hence the law of
\(u_{t_n}\), conditioned on this event, is invariant under permutations
of coordinates \(2,\ldots,n\).  Thus, conditioned on
\(u_1\) and on the multiset
\[
        \{u_2,\ldots,u_n\},
\]
counted with multiplicity, the labels of \(u_2,\ldots,u_n\) are uniform.
The same statement holds for \(v\), independently.

Let \(\mc H\) be the sigma-algebra generated by \(u_1,v_1\) and by the
two multisets
\[
        \{u_2,\ldots,u_n\},
        \qquad
        \{v_2,\ldots,v_n\},
\]
both counted with multiplicity.  Conditioned on \(\mc H\), we may write
the outside contribution as
\[
        \sum_{i=2}^n u_i v_{\Pi(i)}
\]
for a uniform random permutation \(\Pi\) of \(\{2,\ldots,n\}\).  After
reindexing coordinates \(2,\ldots,n\), \cref{lem:permutation-full}
applies with \(m=n-1\) and \(\lambda=D=1/2\).

Since \(\sum_i u_i=\sum_i v_i=0\), we have
\[
        \sum_{i=2}^n u_i=-u_1,
        \qquad
        \sum_{i=2}^n v_i=-v_1 .
\]
Hence \cref{lem:permutation-full} gives, conditioned on \(\mc H\),
\begin{equation}\label{eq:conditional-outside-smoothing}
\mb E_\Pi\exp\left\{
        \frac12\sum_{i=2}^n u_i v_{\Pi(i)}
        +
        \frac12\sum_{i=2}^n u_i^2v_{\Pi(i)}^2
        \right\}                                           
 \le
        \exp\left\{
        \frac{u_1v_1}{2(n-1)}
        +
        C\frac{\norm{u}_2^2\norm{v}_2^2}{n-1}
        \right\}.
\end{equation}

Combining this with the contribution from coordinate \(1\), we get from
\eqref{eq:J-log-full} and \eqref{eq:conditional-outside-smoothing} that
\[
        \mb E[J(u,v)\mid \mc H]
        \le
        \exp\left\{
        \frac12u_1v_1
        +
        \frac12u_1^2v_1^2
        +
        \frac{u_1v_1}{2(n-1)}
        +
        C\frac{\norm{u}_2^2\norm{v}_2^2}{n-1}
        \right\}.
\]
On the good event,
\[
        |u_1|,|v_1|\le\delta,
        \qquad
        \norm{u}_2^2,\norm{v}_2^2\le K L_n(r).
\]
Since \(\delta\le\delta_0\le1/4\), the terms involving \(u_1v_1\) and
\(u_1^2v_1^2\) are bounded above by \(C\delta^2\).  Therefore
\[
        \mb E[J(u,v)\mid \mc H]
        \le
        \exp\left\{
        C\delta^2
        +
        C\frac{K^2L_n(r)^2}{n-1}
        \right\}.
\]
Averaging over \(\mc H\) and using \eqref{eq:J-bound-full}, we obtain
\[
        1+\chi^2(\mu_{n,G}\Vert\pi_n)
        \le
        \exp\left\{
        C\delta^2
        +
        C\frac{K^2L_n(r)^2}{n-1}
        \right\}.
\]

\paragraph{\bf Estimate for \(L_n(r)\).}
It remains to estimate \(L_n(r)^2/(n-1)\).  Since
\[
        \rho_n=1-\frac1{2n}+O(n^{-2}),
        \qquad
        t_n=\ceil{c n\log n},
\]
we have
\[
        \rho_n^{t_n}
        =
        n^{-c/2+o(1)}.
\]
Therefore
\[
        L_n(r)
        =
        \frac{n^2\rho_n^{t_n}}{r^2}
        =
        r^{-2}n^{2-c/2+o(1)},
\]
and hence
\[
        \frac{L_n(r)^2}{n-1}
        =
        O_r(n^{3-c+o(1)}).
\]
This tends to zero because \(c>3\).  Consequently,
\[
        \limsup_{n\to\infty}
        \{1+\chi^2(\mu_{n,G}\Vert\pi_n)\}
        \le e^{C\delta^2},
\]
which proves \eqref{eq:good-chi-full}.
\end{proof}

We now complete the proof of the upper bound.  The assumption
\(c>C_{\mr{BRW}}\) is used through \cref{prop:Linfty-full}.  The
\(\chi^2\) estimate above only requires \(c>3\).

\begin{proof}[Proof of the upper bound in \cref{thm:main-cutoff}]
Fix \(c>C_{\mr{BRW}}\), and set \(t_n=\ceil{c n\log n}\).  By
\cref{prop:hidden-rep} and the scale invariance \(F_{\lambda a}=F_a\),
\[
        \mu_n
        :=
        Q_n^{t_n}(e^{(1)}, \cdot)
        =
        \mb E[F_{\one+u_{t_n}}].
\]
We prove that \(\norm{\mu_n-\pi_n}_{\on{TV}}\to0\).

Fix \(r>0\), \(K<\infty\), and \(0<\delta<\delta_0\), and write
\[
        \mc G_n:=G_n(r,K,\delta).
\]
We first estimate \(\mb P(\mc G_n^c)\).  By \cref{lem:R-tail-full},
\begin{equation}\label{eq:bad-R-full}
        \limsup_{n\to\infty}
        \mb P(R_{t_n}<r)
        \le
        C\sqrt r .
\end{equation}
The \(\ell^\infty\) estimate in \cref{prop:Linfty-full} gives
\[
        \mb P\left(
        R_{t_n}\ge r,\ 
        \norm{u_{t_n}}_\infty>\delta
        \right)
        \longrightarrow0 .
\]
Finally, on \(\{R_{t_n}\ge r\}\),
\[
        \norm{u_{t_n}}_2^2
        =
        \frac{n^2W_{t_n}}{R_{t_n}^2}
        \le
        \frac{n^2W_{t_n}}{r^2}.
\]
Since
\[
        \mb E W_{t_n}=\rho_n^{t_n}\left(1-\frac1n\right),
        \qquad
        L_n(r)=\frac{n^2\rho_n^{t_n}}{r^2},
\]
Markov's inequality gives
\begin{equation}\label{eq:bad-L2-full}
        \mb P\left(
        R_{t_n}\ge r,\ 
        \norm{u_{t_n}}_2^2>K L_n(r)
        \right)
        \le
        \frac1K .
\end{equation}
Combining \eqref{eq:bad-R-full}--\eqref{eq:bad-L2-full}, we obtain
\begin{equation}\label{eq:G-prob-full}
        \limsup_{n\to\infty}
        \mb P(\mc G_n^c)
        \le
        C\sqrt r+\frac1K .
\end{equation}

Choose \(r>0\) sufficiently small and \(K<\infty\) sufficiently large so
that the right-hand side of \eqref{eq:G-prob-full} is strictly less than
\(1\).  For such fixed \(r,K\), we have \(0 < \mb P(\mc G_n) < 1 \) for all
sufficiently large \(n\).  For these \(n\), define
\[
        \mu_{n,G}
        :=
        \mb E[F_{\one+u_{t_n}}\mid \mc G_n], \qquad \mu_{n,B}
        :=
        \mb E[F_{\one+u_{t_n}}\mid \mc G_n^c].
\]
so that
\[
        \mu_n
        =
        \mb P(\mc G_n)\mu_{n,G}
        +
        \mb P(\mc G_n^c)\mu_{n,B}.
\]
Using convexity of total variation and
\(\norm{\mu_{n,B}-\pi_n}_{\on{TV}}\le1\), we get
\begin{equation}\label{eq:TV-split-full}
\begin{aligned}
        \norm{\mu_n-\pi_n}_{\on{TV}}
        &\le
        \mb P(\mc G_n)\norm{\mu_{n,G}-\pi_n}_{\on{TV}}
        +
        \mb P(\mc G_n^c)\norm{\mu_{n,B}-\pi_n}_{\on{TV}}        \\
        &\le
        \norm{\mu_{n,G}-\pi_n}_{\on{TV}}
        +
        \mb P(\mc G_n^c).
\end{aligned}
\end{equation}
The \(\chi^2\)-bound at the beginning of the section gives
\[
        \norm{\mu_{n,G}-\pi_n}_{\on{TV}}
        \le
        \frac12\sqrt{\chi^2(\mu_{n,G}\Vert\pi_n)}.
\]
Therefore, by \cref{prop:good-smoothing-full},
\[
        \limsup_{n\to\infty}
        \norm{\mu_{n,G}-\pi_n}_{\on{TV}}
        \le
        \frac12\sqrt{e^{C\delta^2}-1}.
\]
Combining this with \eqref{eq:G-prob-full} and
\eqref{eq:TV-split-full} yields
\[
        \limsup_{n\to\infty}
        \norm{\mu_n-\pi_n}_{\on{TV}}
        \le
        C\sqrt r+\frac1K+\frac12\sqrt{e^{C\delta^2}-1}.
\]
Now send the parameters to their limiting values in the order
\[
        \delta\downarrow0,
        \qquad
        K\to\infty,
        \qquad
        r\downarrow0 .
\]
It follows that
\[
        \norm{Q_n^{t_n}(e^{(1)},\cdot)-\pi_n}_{\on{TV}}
        =
        \norm{\mu_n-\pi_n}_{\on{TV}}
        \longrightarrow0 .
\]
Finally, \cref{prop:sphere-energy} transfers this convergence from the
energy chain to Kac's walk on the sphere:
\[
        d_n^{(1)}(t_n)\longrightarrow0 .
\]
This proves \eqref{eq:main-upper}.
\end{proof}

\bibliographystyle{amsplain0}
\bibliography{main}

\appendix

\section{Fourth-moment calculations}
\label{app:fourth-moment}

We prove the recursion used in \cref{lem:W-moments-full}.  Recall that
\[
        M_t:=\mb E W_t^2,
        \qquad
        V_t:=\mb E\sum_i\xi_{t,i}^4 .
\]
The exact recursion is
\begin{equation}\label{eq:fourth-matrix-recursion-app}
        \binom{M_{t+1}}{V_{t+1}}
        =
        \begin{pmatrix}
        \dfrac{16n^4-32n^3-39n^2+12n+36}{16n^3(n-1)}
        &
        \dfrac{19n^2+4n+12}{16n^2(n-1)}
        \\[2ex]
        \dfrac{9(n^3+24n^2-8n-24)}{32n^4(n-1)}
        &
        \dfrac{32n^4-61n^3-120n^2-24n-72}{32n^3(n-1)}
        \end{pmatrix}
        \binom{M_t}{V_t}.
\end{equation}
In particular, the matrix in \eqref{eq:fourth-matrix-recursion-app} has
the asymptotic form
\[
        \begin{pmatrix}
        1-\frac1n+O(n^{-2})
        &
        \frac{19}{16n}+O(n^{-2})\\[0.7ex]
        O(n^{-2})
        &
        1-\frac{29}{32n}+O(n^{-2})
        \end{pmatrix}.
\]

Fix \(t\), and suppose that the selected pair is \(\{i,j\}\).  Set
\[
        s_t:=\xi_{t,i}+\xi_{t,j},
        \qquad
        \Delta_t:=\xi_{t,i}-\xi_{t,j}.
\]
The centered update is
\[
        \xi_{t+1,k}
        =
        \xi_{t,k}-\frac{(2B-1)\Delta_t}{n},
        \qquad k\notin\{i,j\},
\]
and
\[
        \xi_{t+1,i}=\xi_{t+1,j}
        =
        \frac{s_t}{2}
        +
        \left(\frac12-\frac1n\right)(2B-1)\Delta_t .
\]
We use only the beta moments
\[
        \mb E(2B-1)^{2\ell+1}=0,
        \qquad
        \mb E(2B-1)^2=\frac12,
        \qquad
        \mb E(2B-1)^4=\frac38 .
\]
In the fixed-pair calculations below, \(\mb E_B\) denotes expectation
only over \(B\), with \(\xi_t\) and the selected pair fixed.

\paragraph{\bf The \(W_t^2\) recursion.}
From the centered update,
\[
        W_{t+1}
        =
        W_t
        +
        (2B-1)\Delta_t s_t
        +
        \left\{
        (2B-1)^2\left(\frac12-\frac1n\right)-\frac12
        \right\}\Delta_t^2 .
\]
Averaging the square over \(B\) gives
\[
\begin{aligned}
        \mb E_B[W_{t+1}^2]
        &=
        W_t^2
        -
        \left(\frac12+\frac1n\right)W_t\Delta_t^2
        +
        \frac12s_t^2\Delta_t^2
        +
        \frac{3n^2+4n+12}{32n^2}\Delta_t^4 .
\end{aligned}
\]
Averaging over the uniformly chosen pair, and using
\(\sum_i\xi_{t,i}=0\), we use the identities
\[
        \sum_{i<j}(\xi_{t,i}-\xi_{t,j})^2=nW_t,
\]
\[
        \sum_{i<j}(\xi_{t,i}-\xi_{t,j})^4
        =
        n\sum_i\xi_{t,i}^4+3W_t^2,
\]
and
\[
        \sum_{i<j}(\xi_{t,i}+\xi_{t,j})^2
        (\xi_{t,i}-\xi_{t,j})^2
        =
        n\sum_i\xi_{t,i}^4-W_t^2.
\]
This gives
\[
\begin{aligned}
        \mb E[W_{t+1}^2\mid \xi_t]
        &=
        \frac{16n^4-32n^3-39n^2+12n+36}{16n^3(n-1)}\,W_t^2 \\
        &\quad+
        \frac{19n^2+4n+12}{16n^2(n-1)}
        \sum_i\xi_{t,i}^4 .
\end{aligned}
\]

\paragraph{\bf The fourth-moment recursion.}
We now compute the same fixed-pair average for
\(\sum_k\xi_{t+1,k}^4\).  For \(k\notin\{i,j\}\), the centered update gives
\[
        \xi_{t+1,k}
        =
        \xi_{t,k}-\frac{(2B-1)\Delta_t}{n}.
\]
Therefore
\[
\begin{aligned}
        \mb E_B
        \sum_{k\notin\{i,j\}}\xi_{t+1,k}^4
        &=
        \sum_{k\notin\{i,j\}}\xi_{t,k}^4
        +
        \frac{3\Delta_t^2}{n^2}
        \sum_{k\notin\{i,j\}}\xi_{t,k}^2
        +
        \frac{3(n-2)}{8n^4}\Delta_t^4                                      \\
        &=
        \sum_{k\notin\{i,j\}}\xi_{t,k}^4
        +
        \frac{3\Delta_t^2}{n^2}
        \left(W_t-\frac{s_t^2+\Delta_t^2}{2}\right)
        +
        \frac{3(n-2)}{8n^4}\Delta_t^4 .
\end{aligned}
\]
For the indices \(i\) and \(j\),
\[
        \xi_{t+1,i}=\xi_{t+1,j}
        =
        \frac{s_t}{2}
        +
        \left(\frac12-\frac1n\right)(2B-1)\Delta_t .
\]
Thus
\[
\begin{aligned}
        \mb E_B\left(\xi_{t+1,i}^4+\xi_{t+1,j}^4\right)
        &=
        \frac{s_t^4}{8}
        +
        \frac32
        \left(\frac12-\frac1n\right)^2
        s_t^2\Delta_t^2
        +
        \frac34
        \left(\frac12-\frac1n\right)^4
        \Delta_t^4 .
\end{aligned}
\]
Since
\[
        \xi_{t,i}^4+\xi_{t,j}^4
        =
        \frac{s_t^4+6s_t^2\Delta_t^2+\Delta_t^4}{8},
\]
the two equations above imply
\[
\begin{aligned}
        \mb E_B\sum_k\xi_{t+1,k}^4
        &=
        \sum_k\xi_{t,k}^4
        +
        \frac{3\Delta_t^2}{n^2}
        \left(W_t-\frac{s_t^2+\Delta_t^2}{2}\right)
        +
        \frac{3(n-2)}{8n^4}\Delta_t^4                         \\
        &\quad+
        \left\{
        \frac32\left(\frac12-\frac1n\right)^2-\frac34
        \right\}s_t^2\Delta_t^2                                \\
        &\quad+
        \left\{
        \frac34\left(\frac12-\frac1n\right)^4-\frac18
        \right\}\Delta_t^4 .
\end{aligned}
\]
Averaging this identity over the selected pair and using the same three
symmetric-sum identities gives
\[
\begin{aligned}
        \mb E\left[\sum_k\xi_{t+1,k}^4\,\middle|\,\xi_t\right]
        &=
        \frac{9(n^3+24n^2-8n-24)}{32n^4(n-1)}\,W_t^2 \\
        &\quad+
        \frac{32n^4-61n^3-120n^2-24n-72}{32n^3(n-1)}
        \sum_i\xi_{t,i}^4 .
\end{aligned}
\]
Taking expectations in the last display and in the \(W_t^2\)-recursion
above proves \eqref{eq:fourth-matrix-recursion-app}.

\section{The permutation estimate}
\label{sub:permutation-estimate}

We use the following consequence of the Hadamard-type inequality for
permanents of Carlen, Lieb and Loss
\cite[Theorem~1.1]{CarlenLiebLoss2006}.  If
\(Z=(Z_{ij})\) is an \(m\times m\) complex matrix, then
\begin{equation}\label{eq:CLL-permanent}
        |\on{per}(Z)|
        \le
        \frac{m!}{m^{m/2}}
        \prod_{i=1}^m
        \left(\sum_{j=1}^m |Z_{ij}|^2\right)^{1/2}.
\end{equation}
Indeed, the theorem in \cite{CarlenLiebLoss2006} is stated with column
norms, and \eqref{eq:CLL-permanent} follows by applying it to
\(Z^{\mathsf T}\).  Here
\[
        \on{per}(Z)
        :=
        \sum_{\pi\in S_m}\prod_{i=1}^m Z_{i,\pi(i)} .
\]

We first derive a centered exponential-moment bound.

\begin{lemma}
\label{lem:rank-one-perm-laplace}
For every \(0\le s_0<\infty\), there are constants
\(\eta_0=\eta_0(s_0)>0\) and \(C=C(s_0)<\infty\) such that the following
holds.  Let \(\alpha,\beta\in\mb R^m\) satisfy
\[
        \sum_i\alpha_i=\sum_i\beta_i=0,
        \qquad
        \norm{\alpha}_\infty\norm{\beta}_\infty\le\eta_0 .
\]
Then, for every \(|s|\le s_0\),
\begin{equation}\label{eq:rank-one-laplace}
        \mb E_\Pi
        \exp\left\{
        s\sum_{i=1}^m\alpha_i\beta_{\Pi(i)}
        \right\}
        \le
        \exp\left\{
        C s^2
        \frac{\norm{\alpha}_2^2\norm{\beta}_2^2}{m}
        \right\},
\end{equation}
where \(\Pi\) is a uniform random permutation of \(\{1,\ldots,m\}\).
\end{lemma}

\begin{proof}
If \(s_0=0\), the assertion is immediate.  Assume \(s_0>0\), and choose
\(\eta_0>0\) so small that
\[
        2s_0\eta_0\le1 .
\]
Let
\[
        Z_{ij}:=\exp\{s\alpha_i\beta_j\}.
\]
Then
\[
\begin{aligned}
        \mb E_\Pi
        \exp\left\{
        s\sum_{i=1}^m\alpha_i\beta_{\Pi(i)}
        \right\}
        &=
        \frac1{m!}\on{per}(Z).
\end{aligned}
\]
By \eqref{eq:CLL-permanent},
\[
\begin{aligned}
        \frac1{m!}\on{per}(Z)
        &\le
        \prod_{i=1}^m
        \left(
        \frac1m\sum_{j=1}^m
        \exp\{2s\alpha_i\beta_j\}
        \right)^{1/2}.
\end{aligned}
\]
For \(|x|\le1\), we use
\[
        e^x\le 1+x+C_1x^2 .
\]
Since
\[
        |2s\alpha_i\beta_j|
        \le
        2s_0\norm{\alpha}_\infty\norm{\beta}_\infty
        \le1
\]
and \(\sum_j\beta_j=0\), we have, for every \(i\),
\[
\begin{aligned}
        \frac1m\sum_{j=1}^m
        \exp\{2s\alpha_i\beta_j\}
        &\le
        1+
        C_1s^2\alpha_i^2
        \frac1m\sum_{j=1}^m\beta_j^2 .
\end{aligned}
\]
Using \(\log(1+x)\le x\), this gives
\[
        \log
        \mb E_\Pi
        \exp\left\{
        s\sum_i\alpha_i\beta_{\Pi(i)}
        \right\}\le
        \frac12
        \sum_{i=1}^m
        C_1s^2\alpha_i^2
        \frac{\norm{\beta}_2^2}{m} =
        C s^2
        \frac{\norm{\alpha}_2^2\norm{\beta}_2^2}{m}.
\]
This proves \eqref{eq:rank-one-laplace}.
\end{proof}

We now prove \cref{lem:permutation-full}.  For reference, the desired
estimate is
\[
 \mb E_\Pi\exp\left\{
 \lambda\sum_{i=1}^m a_i b_{\Pi(i)}
 +D\sum_{i=1}^m a_i^2b_{\Pi(i)}^2
 \right\}
 \le
 \exp\left\{
 \lambda\frac{(\sum_i a_i)(\sum_i b_i)}{m}
 +C\frac{\norm{a}_2^2\norm{b}_2^2}{m}
 \right\},
\]
provided that \(\max_i |a_i|\vee\max_i |b_i|\) is sufficiently small.

\begin{proof}[Proof of \cref{lem:permutation-full}]
If \(m=1\), the estimate follows after increasing \(C\) so that
\(C\ge D\).  We assume \(m\ge2\).

Set
\[
        s_0:=2|\lambda|\vee2D\vee1 .
\]
Let \(\eta_0\) be the threshold in
\cref{lem:rank-one-perm-laplace} for this value of \(s_0\).  Choose
\(0<\delta_0<1\) so small that
\[
        4\delta_0^2\le\eta_0,
        \qquad
        4\delta_0^4\le\eta_0 .
\]
Assume throughout that
\[
        \max_i|a_i|\vee\max_i|b_i|\le\delta_0 .
\]

Write
\[
        \bar a:=\frac1m\sum_i a_i,
        \qquad
        \bar b:=\frac1m\sum_i b_i .
\]
Then
\[
        \sum_i a_i b_{\Pi(i)}
        =
        m\bar a\bar b
        +
        \sum_i (a_i-\bar a)(b_{\Pi(i)}-\bar b).
\]
By Cauchy--Schwarz,
\begin{multline}\label{eq:perm-cauchy-main}
\mb E_\Pi\exp\left\{
        \lambda\sum_i a_i b_{\Pi(i)}
        +
        D\sum_i a_i^2b_{\Pi(i)}^2
        \right\}                                                \\
 \qquad\le
        e^{\lambda m\bar a\bar b}
        \left(
        \mb E_\Pi
        \exp\left\{{2\lambda\sum_i (a_i-\bar a)(b_{\Pi(i)}-\bar b)}\right\}
        \right)^{1/2}
        \left(
        \mb E_\Pi
        \exp\left\{{2D\sum_i a_i^2b_{\Pi(i)}^2}\right\}
        \right)^{1/2}.
\end{multline}

\paragraph{\bf The centered linear term.}
The vectors
\[
        \bigl(a_i-\bar a\bigr)_{i=1}^m,
        \qquad
        \bigl(b_i-\bar b\bigr)_{i=1}^m
\]
have coordinate sum zero.  Also,
\[
        \sum_i(a_i-\bar a)^2\le\norm{a}_2^2,
        \qquad
        \sum_i(b_i-\bar b)^2\le\norm{b}_2^2,
\]
and
\[
        \max_i |a_i-\bar a|\le2\delta_0,
        \qquad
        \max_i |b_i-\bar b|\le2\delta_0.
\]
Thus
\[
        \max_i |a_i-\bar a|\,
        \max_i |b_i-\bar b|
        \le4\delta_0^2
        \le\eta_0 .
\]
Applying \cref{lem:rank-one-perm-laplace} with
\[
        \alpha_i=a_i-\bar a,
        \qquad
        \beta_i=b_i-\bar b,
        \qquad
        s=2\lambda,
\]
gives
\begin{equation}\label{eq:linear-perm-clean}
        \mb E_\Pi
        \exp\left\{2\lambda\sum_i (a_i-\bar a)(b_{\Pi(i)}-\bar b)\right\}
        \le
        \exp\left\{
        C_{\lambda,D}
        \frac{\norm{a}_2^2\norm{b}_2^2}{m}
        \right\}.
\end{equation}

\paragraph{\bf The quadratic term.}
We use
\[
        \sum_i a_i^2b_{\Pi(i)}^2=
        \frac{\norm{a}_2^2\norm{b}_2^2}{m}+
        \sum_i
        \left(a_i^2-\frac{\norm{a}_2^2}{m}\right)
        \left(b_{\Pi(i)}^2-\frac{\norm{b}_2^2}{m}\right).
\]
Both displayed vectors
\[
        \left(a_i^2-\frac{\norm{a}_2^2}{m}\right)_{i=1}^m,
        \qquad
        \left(b_i^2-\frac{\norm{b}_2^2}{m}\right)_{i=1}^m
\]
have coordinate sum zero.  For the first vector,
\[
        \sum_i
        \left(a_i^2-\frac{\norm{a}_2^2}{m}\right)^2
        =
        \sum_i a_i^4-\frac{\norm{a}_2^4}{m}
        \le
        \sum_i a_i^4
        \le
        \delta_0^2\norm{a}_2^2 ,
\]
and the same bound with \(a\) replaced by \(b\) also holds.  Moreover,
since \(\norm{a}_2^2/m\le\delta_0^2\) and
\(\norm{b}_2^2/m\le\delta_0^2\),
\[
        \max_i\left|a_i^2-\frac{\norm{a}_2^2}{m}\right|
        \le2\delta_0^2,
        \qquad
        \max_i\left|b_i^2-\frac{\norm{b}_2^2}{m}\right|
        \le2\delta_0^2 .
\]
Thus the product of the two sup-norms is at most
\[
        4\delta_0^4\le\eta_0 .
\]
Applying \cref{lem:rank-one-perm-laplace} with
\[
        \alpha_i=a_i^2-\frac{\norm{a}_2^2}{m},
        \qquad
        \beta_i=b_i^2-\frac{\norm{b}_2^2}{m},
        \qquad
        s=2D,
\]
yields
\begin{equation}\label{eq:quadratic-perm-clean}
\begin{aligned}
        \mb E_\Pi
        \exp\left\{2D\sum_i a_i^2b_{\Pi(i)}^2\right\}
        &\le
        \exp\left\{
        2D\frac{\norm{a}_2^2\norm{b}_2^2}{m}
        +
        C_{\lambda,D}
        \frac{\delta_0^4\norm{a}_2^2\norm{b}_2^2}{m}
        \right\}                                                \\
        &\le
        \exp\left\{
        C_{\lambda,D}
        \frac{\norm{a}_2^2\norm{b}_2^2}{m}
        \right\}.
\end{aligned}
\end{equation}

Combining \eqref{eq:perm-cauchy-main},
\eqref{eq:linear-perm-clean}, and
\eqref{eq:quadratic-perm-clean}, and using
\[
        m\bar a\bar b
        =
        \frac{(\sum_i a_i)(\sum_i b_i)}{m},
\]
proves \eqref{eq:perm-bound-full}.
\end{proof}

\section{Numerical verification}
\label{app:numerical}

This appendix proves \cref{lem:numerical-choice}.  Recall that
\[
        m(p)=\frac{2\Gamma(p+1/2)}{\sqrt\pi\,\Gamma(p+1)},
        \qquad
        \sigma_p=\frac{6+13p}{64},
\]
and
\[
        C_{\mr{BRW}}
        =
        \inf_{p>1}\frac{p}{2(1-m(p))}.
\]

The script \path{numerics.py} certifies the interval estimates below using
Arb ball arithmetic.  It encloses the gamma ratio defining \(m(p)\), using
the monotonicity of \(m\) on the interval under consideration, and it
encloses a digamma difference by combining a finite partial sum with
rigorous integral bounds for the tail.

Let \(\psi=(\log\Gamma)'\) denote the digamma function, and set
\[
        D(p):=\psi(p+1)-\psi(p+1/2).
\]
Then
\[
        D(p)
        =
        \sum_{k=0}^{\infty}
        \left(
        \frac1{k+p+1/2}-\frac1{k+p+1}
        \right)
        =
        \sum_{k=0}^{\infty}
        \frac{1}{2(k+p+1/2)(k+p+1)}.
\]
In particular, \(D(p)>0\) for \(p>1\).  Since
\[
        m'(p)
        =
        m(p)\{\psi(p+1/2)-\psi(p+1)\}
        =
        -m(p)D(p),
\]
the function \(m\) is strictly decreasing on \((1,\infty)\).

For the numerical certification, we write
\[
        D(p)
        =
        \sum_{k=0}^{N-1}
        \frac{1}{2(k+p+1/2)(k+p+1)}
        +R_N(p).
\]
The summand is positive and decreasing in \(k\), and therefore
\[
        \int_N^\infty
        \frac{du}{2(u+p+1/2)(u+p+1)}
        \le R_N(p)
        \le
        \int_{N-1}^\infty
        \frac{du}{2(u+p+1/2)(u+p+1)} .
\]
The integrals are explicit:
\[
        \int_M^\infty
        \frac{du}{2(u+p+1/2)(u+p+1)}
        =
        \log\frac{M+p+1}{M+p+1/2}.
\]

We first locate the minimizer of
\[
        p\mapsto \frac{p}{2(1-m(p))}.
\]
Since
\[
        \frac{d}{dp}\left\{\frac{p}{2(1-m(p))}\right\}
        =
        \frac{1-m(p)-pm(p)D(p)}{2(1-m(p))^2},
\]
define
\[
        H(p):=1-m(p)-pm(p)D(p).
\]
For \(p>1\), the monotonicity of \(m\) gives \(m(p)<m(1)=1\), and hence
the derivative of \(p/[2(1-m(p))]\) has the same sign as \(H(p)\).

We now show that \(H\) is strictly increasing on \((1,\infty)\).
Differentiating and using \(m'(p)=-m(p)D(p)\) gives
\[
        H'(p)
        =
        pm(p)\{D(p)^2-D'(p)\}.
\]
On the other hand, differentiating the locally uniformly convergent series
for \(D\) gives
\[
        D'(p)
        =
        \sum_{k=0}^{\infty}
        \left(
        -\frac1{(k+p+1/2)^2}
        +
        \frac1{(k+p+1)^2}
        \right)
        <0.
\]
Therefore \(H'(p)>0\) for every \(p>1\).

Finally,
\[
        \frac{p}{2(1-m(p))}\to\infty
        \qquad\text{as }p\downarrow1
        \quad\text{and as }p\to\infty.
\]
The first divergence follows from \(m(1)=1\) and \(m'(1)<0\), while the
second follows from the gamma-ratio asymptotic \(m(p)\to0\).  Thus, if
\(H(a)<0<H(b)\), then \(H\) has a unique zero \(p_*\in(a,b)\).  The
function \(p/[2(1-m(p))]\) is decreasing on \((1,p_*)\) and increasing
on \((p_*,\infty)\), so \(p_*\) is its unique global minimizer.

The interval computation verifies this sign change on
\[
        I=[2.41258,\,2.41259].
\]
More precisely, the script certifies
\[
        H(2.41258)<0<H(2.41259).
\]
It follows that the unique minimizer lies in \(I\), and hence
\[
        C_{\mr{BRW}}
        =
        \inf_{p\in I}\frac{p}{2(1-m(p))}.
\]
Subdividing \(I\) and applying interval arithmetic on each subinterval
gives
\begin{equation}\label{eq:numerical-certificate}
        C_{\mr{BRW}}
        \in[3.89160134,\,3.89160138].
\end{equation}
In particular \(C_{\mr{BRW}}>3\).  On the same interval, the interval
calculation gives
\begin{equation}\label{eq:numerical-certificate-aux}
        \sup_{p\in I}\frac{1+p/2}{\sigma_p}<3.779154,
        \qquad
        \inf_{p\in I}\left\{2(1-m(p))-\sigma_p\right\}>0.03613 .
\end{equation}

We now prove \cref{lem:numerical-choice}.  Fix
\(c>C_{\mr{BRW}}\).  Since
\[
        C_{\mr{BRW}}
        =
        \inf_{p\in I}\frac{p}{2(1-m(p))},
\]
there exists \(p\in I\subset(2,4)\) such that
\[
        \frac{p}{2(1-m(p))}<c.
\]
For this same \(p\), \eqref{eq:numerical-certificate} and
\eqref{eq:numerical-certificate-aux} give
\[
        \frac{1+p/2}{\sigma_p}
        <3.779154<3.89160134\le C_{\mr{BRW}}<c
\]
and
\[
        2(1-m(p))>\sigma_p.
\]
Therefore
\[
        c>\frac{p}{2(1-m(p))},
        \qquad
        c>\frac{1+p/2}{\sigma_p},
        \qquad
        2(1-m(p))>\sigma_p,
\]
which proves \cref{lem:numerical-choice}.

\end{document}